\documentclass[a4paper,10pt]{amsart}
\usepackage[utf8]{inputenc}
\usepackage{amssymb}
\usepackage{amsmath}
\usepackage{bbm}
\usepackage{mathrsfs}
\usepackage[all]{xy}
\usepackage{manfnt}
\usepackage{pigpen}
\usepackage{bm}
\usepackage{graphicx}
\usepackage{enumitem}
\usepackage[backref=page]{hyperref}
\usepackage[left=2cm,right=2cm,top=2cm,bottom=2cm]{geometry}
\usepackage{tikz}
\title[Fundamental structure for Lie algebra homology]{On fundamental structure underlying Lie algebra homology with coefficients tensor products of the adjoint representation}
\author{Geoffrey Powell}
\address{Univ Angers, CNRS, LAREMA, SFR MATHSTIC, F-49000 Angers, France}
\email{Geoffrey.Powell@math.cnrs.fr}
\urladdr{https://math.univ-angers.fr/~powell/}
\subjclass[2000]{}
\newtheorem{THM}{Theorem}

\newtheorem{COR}[THM]{Corollary}
\newtheorem{thm}{Theorem}[section]
\newtheorem{prop}[thm]{Proposition}

\newtheorem{cor}[thm]{Corollary}
\newtheorem{lem}[thm]{Lemma}

\theoremstyle{definition}
\newtheorem{defn}[thm]{Definition}
\newtheorem{exam}[thm]{Example}

\theoremstyle{remark}
\newtheorem{rem}[thm]{Remark}
\newtheorem{nota}[thm]{Notation}

\renewcommand{\phi}{\varphi}
\renewcommand{\hom}{\mathrm{Hom}}
\newcommand{\sym}{\mathfrak{S}}
\newcommand{\gr}{\mathbf{gr}}

\newcommand{\qmod}{\mathtt{Mod}_\rat}
\newcommand{\f}{\mathcal{F}}
\newcommand{\nat}{\mathbb{N}}

\newcommand{\op}{^\mathrm{op}}
\newcommand{\ob}{\mathrm{Ob}\hspace{2pt}}
\newcommand{\fb}{{\bm{\Sigma}}}
\newcommand{\fs}{\bm{\Omega}}
\newcommand{\lie}{\mathfrak{Lie}}
\newcommand{\ass}{\mathfrak{Ass}}
\newcommand{\uass}{\mathfrak{Ass}^u}
\newcommand{\smodug}{{{}_{\rat\fb}\mathtt{Mod}}}
\newcommand{\smodopug}{{\mathtt{Mod}_{\rat\fb}}}
\newcommand{\fbbimod}{{{}_{\rat\fb}\mathtt{Mod}_{\rat\fb}}}
\newcommand{\cat}{\mathbf{Cat}\hspace{1pt}}
\newcommand{\flie}{\f_{\lie}}
\newcommand{\g}{\mathfrak{g}}
\newcommand{\id}{\mathrm{Id}}
\newcommand{\mut}{\widetilde{\mu}}
\newcommand{\lmod}{{}_{\catlie}\mathtt{Mod}}
\newcommand{\rmod}[1][A]{\mathtt{Mod}_{\catlie}}
\newcommand{\bimod}{{}_{\catlie}\mathtt{Mod}_{\catlie}}
\newcommand{\catlie}{{\cat\lie}}
\newcommand{\rat}{\mathbb{Q}}
\newcommand{\liealg}{\mathsf{Lie}}
\newcommand{\dgcat}{\mathbb{C}\lie}
\newcommand{\homol}{\mathbf{H}}
\newcommand{\lHmod}{{}_{\homol_0} \mathtt{Mod}}
\newcommand{\dgmod}{{}_{\dgcat} \mathtt{DGMod}}
\newcommand{\ce}{\mathsf{CE}}
\numberwithin{equation}{section}
\begin{document}

\begin{abstract}
This paper exhibits fundamental structure underlying Lie algebra homology with coefficients in tensor products of the adjoint representation, mostly focusing upon the case of free Lie algebras. 

The main result yields a DG category that is constructed from the PROP associated to the Lie operad. Underlying this is a two-term complex of bimodules over this PROP; it is  a quotient of the universal Chevalley-Eilenberg complex.

The homology of this DG category  is intimately related to outer functors over free groups (introduced in earlier joint work with Vespa). This uses the author's previous results relating functors on free groups to representations of the PROP associated to the Lie operad.

This gives a  direct algebraic explanation as to why the degree one homology  should correspond to an outer functor. Hitherto, the only known argument relied upon the relationship with the higher Hochschild homology functors that arise from the work of Turchin and Willwacher.
\end{abstract}

\maketitle

\section{Introduction}
\label{sect:intro}

This paper can be viewed in the general context of understanding the natural structure of the family of  Lie algebra homology groups  $H_* (\g ; \g^{\otimes n})$, for $n \in \nat$ and a $\rat$-Lie algebra $\g$,  with coefficients given by the  iterated tensor product of the adjoint representation.  Not only does the symmetric group $\sym_n$ act naturally on $H_* (\g; \g^{\otimes n})$ via the place permutation action on the coefficients, but the family forms a left $\catlie$-module, where $\catlie$ is the $\rat$-linear category associated to the Lie operad. This action is induced by that on the coefficients (explicitly, $\underline{\g} : n \mapsto \g^{\otimes n}$ forms a left $\catlie$-module), together with the fact that this structure is compatible with the $\g$-actions.

Here,  we mostly focus upon the universal case, namely the free Lie algebras $\liealg (V)$, considered as a functor of the $\rat$-vector space $V$. (In the body of the paper, in order to treat all the structure that appears, we work directly with the Lie operad and structures derived from it.) 

Our interest in the case $\g = \liealg (V)$ stems from the relationship between the Lie algebra homology
 $$H_* (\liealg  (V); \liealg(V) ^{\otimes \bullet}) $$
 (considered as a functor of $V$) and the higher Hochschild homology functors that arise from the work of Turchin and Willwacher \cite{MR3982870}. This relationship follows from the author's work with Vespa \cite{PV}, using the interpretation of analytic contravariant functors on free groups in terms of left $\catlie$-modules  \cite{P_analytic}, together with the interpretation of outer functors on free groups in this framework \cite{P_outer}. (An outer functor is a functor on the category of free groups on which all inner automorphism groups act trivially.)

For $\g = \liealg (V)$, the only non-trivial homology groups are in homological dimensions $0$ and $1$; more particularly, the homology with coefficients $\liealg(V)^{\otimes n}$ is calculated as that of the chain complex 
\begin{eqnarray}
\label{eqn:cx_V}
\liealg (V)^{\otimes n} \otimes V 
\rightarrow 
\liealg(V)^{\otimes n},
\end{eqnarray}
which is natural in $V$. 

In order to treat all structure, this is interpreted here via the Schur correspondence as a chain complex 
\begin{eqnarray}
\label{eqn:mut}
 \delta^{(1)} \catlie 
\cong 
 \catlie \odot I
 \stackrel{\mut^{(1)}}{ \longrightarrow} \catlie 
\end{eqnarray}
that encodes (\ref{eqn:cx_V}) for all $n \in \nat$. (Here $\catlie \odot I$ is the convolution product of $\catlie$ with the identity operad $I$; the functor $\delta^{(1)}$ and the natural transformation $\mut^{(1)}$ are introduced in Section \ref{sect:complex}.) 

The main purpose of this paper is to exhibit the rich structure of the chain complex (\ref{eqn:mut}). The first step is showing that there is a natural $\catlie$-bimodule structure on $\delta^{(1)} \catlie$. Whereas the left action is given for the reasons above, the right action is slightly more subtle (see Proposition \ref{prop:delta1_catlie_bimodule}). Then, in Proposition \ref{prop:mut1_bimodule}, it is proved that $\mut^{(1)}$ is a morphism of $\catlie$-bimodules. 

That (\ref{eqn:mut}) might be a morphism of $\catlie$-bimodules was first suggested by the fact that the `universal' Chevalley-Eilenberg complex $\ce_\lie$ for $\lie$ with coefficients in $\catlie$ is a complex of $\catlie$-bimodules (the construction of $\ce_\lie$ is reviewed in Section \ref{sect:ce}). The key property of $\ce_\lie$ is that, for $\g$ a Lie algebra, $\ce_\lie \otimes_\catlie \underline{\g}$ is the Chevalley-Eilenberg complex that calculates $H_* (\g; \underline{\g})$, where $\underline{\g} : n \mapsto \g^{\otimes n}$ as above.

The $\catlie$-bimodule structure on $\delta^{(1)} \catlie$ allows the formation of 
\[
\dgcat := \catlie \oplus \delta^{(1)} \catlie [1],
\]
which is considered as a category enriched in $\nat$-graded $\rat$-vector spaces. It can be viewed as the square-zero extension of $\catlie$ by $\delta^{(1)}\catlie$, with appropriate homological grading. 

The differential $\mut^{(1)}$ is compatible with the category structure:

\begin{THM}
\label{THM:A}
(Theorem \ref{thm:DG_category}.)
The category $\dgcat = (\catlie \oplus \delta^{(1)} \catlie[1], \mut^{(1)})$ equipped with the differential $\mut^{(1)}$ is a differential graded category. 
\end{THM}

Writing the homology of $\dgcat$ as $\homol$, it follows immediately that $\homol$ is a category enriched in $\nat$-graded $\rat$-vector spaces; in particular, $\homol_0$ is a $\rat$-linear category (see Proposition \ref{prop:homol}). 

Now, the category of $\homol_0$-modules identifies with the category $\flie^\mu$ introduced in \cite{P_outer} that models the category of outer functors on free groups (see Theorem \ref{thm:fliemu_homol0}). From this, one deduces immediately:

\begin{COR}
\label{COR}
(Corollary \ref{cor:H1_outer}.)
Considered as a left $\catlie$-module, $\homol_1$ belongs to $\flie^\mu$. 
\end{COR}

Via the theory of \cite{P_analytic}, this can be paraphrased as stating that $\homol_1$ corresponds to an outer functor. This is of interest in that it gives a direct {\em algebraic} explanation as to why this should be true. Hitherto, the only known argument relied on the interpretation in terms of higher Hochschild homology, where one can appeal to the fact that the conjugation action on the classifying space of a group in {\em unpointed} spaces is homotopically trivial. 
 Theorem \ref{THM:A} should be seen as the fundamental underlying result, providing much more structure, encapsulated in a simple framework.

If one forgets structure, then the underlying complex of $\catlie$-bimodules of $\dgcat$ is related to the universal Chevalley-Eilenberg complex $\ce_\lie$. Explicitly: 

\begin{THM}
\label{THM:B}
(Cf. Theorem \ref{thm:ce_dgcat}.)
There is a surjective  morphism of complexes of $\catlie$-bimodules 
\[
\ce_\lie \twoheadrightarrow \dgcat
\]
which is a weak equivalence. Moreover, $\ce_\lie$ has terms that are projective as right $\catlie$-modules.
\end{THM}

This result serves to make explicit the relationship between Lie algebra homology and an approximation that is obtained by using $\dgcat$ (see Example \ref{exam:lie_alg_hom_vs_dgcat_hom}). Moreover, it is clear that this is not specific to working with Lie algebras: one can consider a left $\catlie$-module $M$ and the induced surjection of chain complexes 
\[
\ce_\lie \otimes_\catlie M 
\twoheadrightarrow \dgcat\otimes_\catlie M.
\] 
(See Corollary \ref{cor:nat_trans_ce_dgcat_homol} and Proposition \ref{prop:properties_dgcat_otimes}.)

Now, as explained in Section \ref{subsect:exam}, the homology $H_* (\ce_\lie \otimes_\catlie M)$ gives an approximation to the left derived functors of
\[
M \mapsto \homol_0 \otimes_{\catlie} M.
\]
These left derived functors are important in relation to the theory of \cite{P_outer}. This is significant, since the functors $M \mapsto H_* (\ce_\lie \otimes_\catlie M)$ are calculable for the fundamental family given by $M= \rat \sym_n$ (for $n \in \nat$), considered as a left $\catlie$-module supported on $\mathbf{n}$ (see Proposition \ref{prop:H_ce_QSn}). Consequences of this will be developed elsewhere.

There is another interesting and conceptual relationship between $\homol_1$ and $\homol_0$ that was implicit in \cite{PV} (in the following statement, $\fb$ is the category of finite sets and bijections): 

\begin{THM}
(Theorem \ref{thm:syzygy}.)
The underlying complex of $\dgcat$ provides a projective resolution of $\homol_0$ in left $\catlie$, right $\rat\fb$-modules.

Hence, the exact sequence 
\[
0
\rightarrow
\homol_1 
\rightarrow
\catlie \odot I
\rightarrow 
\catlie 
\rightarrow 
\homol_0
\rightarrow 
0
\]
restricted to left $\catlie$, right $\rat\fb$-modules exhibits $\homol_1$ as the second syzygy of $\homol_0$.
\end{THM}

Theoretically this is a pleasing statement, giving a relationship between the underlying structures of $\homol_0$ and $\homol_1$. However, on the computational level it does not give any additional information. From this viewpoint, Corollary \ref{COR} is more significant, together with the further structure derived from Theorem \ref{THM:A}.

\tableofcontents

\section{Background}
\label{sect:background}

The purpose of this section is to review the theory of \cite{P_analytic} that give models for the appropriate functors on free groups. This uses the category of left $\catlie$-modules, where $\catlie$ is the category associated to the Lie operad $\lie$. Then we review the theory of \cite{P_outer}, which identified the structure corresponding to outer functors on free groups, introduced in the author's work with Vespa, \cite{PV}. 

Throughout we work over the field $\rat$ of rational numbers and $\qmod$ denotes the category of $\rat$-vector spaces.  
All undecorated tensor products $\otimes$ denote $\otimes_\rat$.

\subsection{$\fb$-modules and more}
The category of finite sets and bijections is denoted $\fb$, which has small skeleton given by the objects $\mathbf{n} = \{1, \ldots , n \}$, for $n \in \nat$. Since $\fb$ is a groupoid, there is an isomorphism of categories $\fb\cong \fb\op$ induced by taking the inverse. The $\rat$-linearization of $\fb$ is denoted $\rat \fb$. 

The category $\smodug$ of $\rat \fb$-modules is the category of $\rat$-linear functors from  $\rat\fb$ to $\qmod$ (this is equivalent to the category of functors from $\fb$ to $\qmod$). Likewise one has the category $\smodopug$ of  $\rat \fb\op$-modules (or {\em right} $\rat \fb$-modules) and $\fbbimod$, the category of $\rat\fb$-bimodules. Clearly $\smodopug$ is equivalent to $\smodug$, via the isomorphism $\fb \cong \fb\op$.

\begin{rem}
The terminology $\rat\fb$-modules is used here for consistency, since we also work with categories of modules over more general $\rat$-linear categories. (In the literature, $\smodug$ is often known as the category of $\fb$-modules.)
\end{rem}

Recall that the functor $\otimes_\fb : \smodopug \times \smodug \rightarrow \qmod$ is defined for a right $\rat\fb$-module $M$ and a left $\rat\fb$-module $N$ by 
$
M \otimes_\fb N = \bigoplus_{n \in \nat} M (\mathbf {n}) \otimes_{\sym_n} N (\mathbf{n})$,
 where  $\sym_n = \mathrm{Aut} (\mathbf{n})$.

\begin{exam}
For $V$ a $\rat$-vector space, one has the associated $\rat\fb$-module $\underline{V}$ given by $\underline{V} (\mathbf{n}) := V^{\otimes n}$, $n \in \nat$, with the symmetric group $\sym_n$ acting by place permutations. This construction is natural in $V$. 

Then, given a right $\rat\fb$-module $M$, one can form $M \otimes_\fb \underline{V}$, considered as an endofunctor of $\rat$-vector spaces:
\[
V \mapsto
M \otimes_\fb \underline{V} = \bigoplus_{n \in \nat} M(\mathbf{n}) \otimes_{\sym_n} V^{\otimes n}.
\]
This is the Schur functor $V \mapsto M(V)$  associated to $M$; it is an analytic functor of $V$, expressed as the sum of the homogeneous polynomial functors $V \mapsto M(\mathbf{n}) \otimes_{\sym_n} V^{\otimes n}$, $n \in \nat$. One can recover $M$ from this Schur functor; we refer to this as the Schur correspondence. (See \cite[Appendix I.A]{MR3443860} for this and more.)
\end{exam}

The convolution product of left $\rat\fb$-modules is denoted $\odot$;  for two $\rat\fb$-modules $M, N$ and a finite set $S$, this is given by 
\[
M \odot N\  (S) = \bigoplus_{S_1 \amalg S_2 = S} M(S_1) \otimes N(S_2), 
\]
where the sum is over ordered decompositions of $S$ into two subsets.  This gives a symmetric monoidal structure $(\smodug, \odot, \rat)$, where $\rat$ is considered as a left $\rat\fb$-module supported on $\mathbf{0} = \emptyset$. Likewise one has the symmetric monoidal structure $(\smodopug, \odot, \rat)$.

\begin{exam}
\label{exam:fb_schur}
Let $M$, $N$ be right $\rat\fb$-modules; the Schur functor associated to the convolution product $M \odot N$ is naturally isomorphic to the tensor product of the Schur functors of $M$ and $N$: 
\[
(M \odot N) (V)   \cong 
M(V)  \otimes 
N(V).
\]

Write $\circ$ for the composition product for $\rat\fb\op$-modules, which gives the monoidal structure $(\smodopug, \circ, I)$, where $I$ is $\rat$ supported on $\mathbf{1}$  (see \cite[Section 5.2]{LV}, for example). 
 The Schur functor associated to $I$ is the identity functor $I(V) = V$. The Schur functor associated to $M \circ N$ corresponds to the composite:
 \[
 (M \circ N) (V) \cong M (N (V)).
 \]
 
By the Schur correspondence, these identifications encode the monoidal structure  $(\smodopug, \circ, I)$. Since operads are, by definitions, monoids in $(\smodopug, \circ, I)$, the Schur correspondence is a useful tool in studying these; for example, this is employed in Example \ref{exam:gbar} below.
\end{exam}

\subsection{Introducing $\cat\lie$}

Write $\lie$ for the Lie operad (in $\rat$-vector spaces) and $\catlie$ for the associated PROP. Thus $\catlie$ is a $\rat$-linear category with set of objects $\nat$, and with symmetric monoidal structure $\boxplus$ that corresponds to addition on objects. Explicitly, for $m,n \in \nat$, one has 
\[
\catlie (m,n) = \bigoplus_{f \in \hom_{\fs}(\mathbf{m}, \mathbf{n})} \bigotimes_{i=1} ^n \lie (|f^{-1} (i)|),
\]
where the sum is over $\hom_{\fs}(\mathbf{m}, \mathbf{n})$, the set of  surjective set maps from $\mathbf{m}$ to $\mathbf{n}$. In particular, $\catlie (m,n)$ is zero if $n>m$ and is isomorphic (as a $\rat$-algebra) to $\rat \sym_n$ if $m=n$. Thus, forgetting structure, $\catlie$ is a $\rat\fb$-bimodule ($\sym_n$ acts on the left on $\catlie (m,n)$ and $\sym_m$ on the right, and these actions commute).

The category of left $\catlie$-modules is denoted $\lmod$. This is equivalent to the category of $\rat$-linear functors from $\catlie$ to $\qmod$. Likewise one has the category $\rmod$ of right $\catlie$-modules  and the category $\bimod$ of $\catlie$-bimodules.

\begin{exam}
\label{exam:catlie_bimodule}
By restriction of structure,  $\catlie$ is canonically a $\catlie$-bimodule.
\end{exam}

\begin{rem}
The category of right $\catlie$-modules is equivalent to the category of right $\lie$-modules defined with respect to the monoidal structure $(\smodopug, \circ , I)$ (see \cite[Proposition 1.2.6]{KM}).
\end{rem}

\begin{exam}
\label{exam:gbar}
For $\mathfrak{g}$ a Lie algebra, the left $\rat\fb$-module $\underline{\g}$ given by $n \mapsto \g^{\otimes n}$ has a canonical left $\catlie$-module structure induced by the Lie algebra structure of $\g$ (see \cite[Proposition 5.4.2]{LV}, for example).

Consider $\catlie$ as a left $\catlie$, right $\rat\fb$ bimodule (i.e., restricting the right $\catlie$-module structure given by Example \ref{exam:catlie_bimodule} along $\rat \fb \rightarrow \catlie$ induced by the unit $I \rightarrow \lie$). One can then form the Schur functor 
 $V \mapsto  \catlie \otimes_\fb \underline{V}$; this takes values in left $\catlie$-modules. 
 This identifies  with the left $\catlie$-module $\underline{\liealg(V)}$, where $\liealg (V)$ is the free Lie algebra on $V$ (which has underlying  functor $\lie \otimes _\fb \underline{V}$). In particular, for $n\in \nat$:
 \[
 \catlie (-, n) \otimes_\fb \underline{V} \cong \liealg (V)^{\otimes n}.
\]

This can be used to understand the right $\catlie$-module structure as follows. Composition in $\catlie$ corresponds to a morphism of $\fb$-bimodules 
\[
\catlie \otimes_\fb \catlie \rightarrow \catlie,
\]
(variance dictates how $\otimes_\fb$ is formed). Applying $- \otimes_\fb \underline{V}$ this gives 
\[
(\catlie \otimes_\fb \catlie) \otimes_\fb \underline{V} \cong \catlie \otimes _\fb \underline{\liealg (V)}  \rightarrow 
\catlie\otimes_\fb \underline{V}
\cong 
\underline{\liealg (V)},
\]
using associativity to give the first isomorphism. Now, $\catlie \otimes _\fb \underline{\liealg (V)}$ is naturally isomorphic to $\underline{\liealg (\liealg (V))}$, by repeating the identification, hence the structure map is a morphism of left $\catlie$-modules 
\[
\underline{\liealg (\liealg (V)) } \rightarrow \underline{\liealg(V)}. 
\]
This is  the morphism of left $\catlie$-modules induced by the canonical Lie algebra map
\[
\liealg (\liealg (V)) \rightarrow \liealg (V)
\]
that corresponds to the composition $\lie \circ \lie \rightarrow \lie$ of the Lie operad. 
\end{exam}

\subsection{Convolution for right $\catlie$-modules}

The convolution product on $\smodopug$ upgrades to define a symmetric monoidal structure   $(\rmod, \odot, \rat)$   (see \cite[Proposition 1.6.3]{KM}, for example, which references \cite{MR1665330}).

\begin{exam}
\label{exam:conv_right_catlie}
For a right $\catlie$-module $M$, after passage to the associated Schur functor, similarly to the analysis in Example \ref{exam:gbar}, the structure map is encoded by a natural transformation (with respect to $V$)
\[
M (\liealg (V)) \rightarrow M(V)
\]
that is unital and associative with respect to the right $\lie$-action, in the appropriate sense.

For $N$ a second right $\catlie$-module, the convolution product (as right $\catlie$-modules) $M \odot N$ has structure natural transformation 
\[
(M \odot N) (\liealg (V)) \rightarrow (M \odot N) (V) 
\]
that is given by the tensor product of the structure maps $M (\liealg (V)) \rightarrow M(V)$ and $N (\liealg (V)) \rightarrow N(V)$, using the natural isomorphisms $(M \odot N) (\liealg (V)) \cong M(\liealg (V)) \otimes N(\liealg (V))$ and $(M \odot N) (V) \cong M(V) \otimes N(V)$.
\end{exam}

\begin{rem}
\label{rem:conv_left_catlie_modules}
There is also a convolution product for left $\catlie$-modules; this is used in \cite{P_analytic} and \cite{P_outer}. This is not exploited here.
\end{rem}

The following structure is sufficient for current purposes:

\begin{prop}
\label{prop:convolution_bimod_rmod}
The convolution product $\odot$ on $\rmod$ induces a functor 
\[
\bimod \times \rmod \rightarrow \bimod
\]
where, for $n \in \nat$, a bimodule $B$ and a right module $M$, $B \odot M \ (-,n):= B (-,n) \odot M$ as a right $\catlie$-module, with the left $\catlie$-module structure induced by that  of $B$. 
\end{prop}

\begin{proof}
Consider the respective Schur functors and their structure morphisms, as in Example \ref{exam:conv_right_catlie}. For $B$, one has  $B \otimes _\fb \underline{V}$; this takes values naturally in left $\catlie$-modules and the structure morphism 
\[
B \otimes _\fb \underline{\liealg(V)}
\rightarrow
B \otimes _\fb \underline{V}
\]
is a morphism of left $\catlie$-modules.  For $M$, one has the structure morphism $M (\liealg (V)) \rightarrow M(V)$ (in $\rat$-vector spaces). 

One can form $(B \otimes _\fb \underline{V}) \otimes M(V)$ and equip this with the tensor product of the above structure morphisms giving 
\[
(B \otimes _\fb \underline{\liealg(V)}) \otimes M(\liealg (V)) 
\rightarrow 
(B \otimes _\fb \underline{V}) \otimes M(V).
\]
This is natural in $V$ and unital and associative as required. Moreover, it is clearly a morphism of left $\catlie$-modules for the action arising from $B$. Hence this encodes a $\catlie$-bimodule. One checks that this corresponds to the structure in the statement.  
\end{proof}

\subsection{The shift functor $\delta$}

The symmetric monoidal structure $\boxplus$ of $\catlie$ induces the $\rat$-linear functor 
\[
- \boxplus 1 : \catlie \rightarrow \catlie.
\]
Thinking of left $\catlie$-modules as $\rat$-linear functors from $\catlie$ to $\qmod$, 
precomposition with $- \boxplus 1$ gives the exact functor 
\[
\delta : \lmod \rightarrow \lmod. 
\]
Thus, for a left $\catlie$-module $M$ and $n \in \nat$,  $(\delta M) (n) = M (n+1)$. (Here and elsewhere, for an $\rat\fb$-module $M$ and for $s \in \nat$, $M(s)$ is synonymous with $M(\mathbf{s})$.)

If $B$ is a $\catlie$-bimodule, $\delta$ can be applied with respect to the left $\catlie$-module structure and this yields $\delta B$, which is a $\catlie$-bimodule, giving  the functor 
\[
\delta : \bimod \rightarrow \bimod.
\]

\begin{lem}
\label{lem:delta_catlie}
There is an isomorphism of $\catlie$-bimodules
\[
\delta \catlie \cong \catlie \odot \lie,
\]
using the structure provided by Proposition \ref{prop:convolution_bimod_rmod} on the right hand side.
\end{lem}

\begin{proof}
This is a straightforward verification using the construction of $\catlie$.
\end{proof}

\subsection{The morphisms $\mu( n)$}

The purpose of this section is to introduce the important sequence of morphisms 
\[
\mu (n) \in \catlie (n+1, n)
\]
for $n \in \nat$.  These were introduced and exploited in \cite{P_outer}.

Recall that $\lie (1) = \rat$, with canonical generator given by the unit; $\lie(2) \cong \mathrm{sgn}_2$, the signature representation of $\sym_2$, with underlying vector space $\rat$, generated by $[-,-]$. 

By construction, $\catlie (n+1,n) = \bigoplus_{f \in \hom_{\fs} (\mathbf{n+1}, \mathbf{n})} \bigotimes \lie (|f^{-1} (i)|)$. Now, a surjection $f : \mathbf{n+1} \twoheadrightarrow \mathbf{n}$ has exactly one fibre of cardinal two, all the others having cardinal one. Using the above generators, the summand indexed by $f$ is $\rat$, with  a canonical generator that we denote by $\mu_f (n)$. 

\begin{defn}
\label{defn:lambda}
Set $\mu (0)=0$ and, for $n \geq 1$, define $\mu (n)\in \catlie (n+1, n)$ to be the element
\[
\mu (n) := \sum_{\substack{f \in \hom_{\fs} (\mathbf{n+1}, \mathbf{n}) \\ {f|_{\mathbf{n}} = \mathrm{id}_{\mathbf{n}}}}} \mu_f (n),
\]
where the sum is taken over surjections $\mathbf{n+1} \cong \mathbf{n} \amalg \mathbf{1} \twoheadrightarrow \mathbf{n}$ that are the identity when restricted to $\mathbf{n} \subset \mathbf{n+1}$.
\end{defn}

\begin{rem}
\label{rem:box_notation}
It is convenient to use `box notation' to represent the elements $\mu (n)$, using standard conventions for representing morphisms in PROPs (with entries at the top and exits at the bottom). Such notation is used in the study of Jacobi diagrams (see  \cite[Example 3.2]{MR4321214} for example). 

Using this, $\mu (3)$ is represented by:
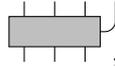
\begin{center}
 \begin{tikzpicture}[scale = .2]
 \draw (1,1) -- (1,-3);
 \draw (3,1) -- (3,-3);
 \draw (5,1) -- (5,-3);
\draw [rounded corners] (7,1) -- (7,-1) -- (6,-1);
 \draw [fill= lightgray] (0,0) -- (6,0) -- (6, -2)  -- (0, -2) -- cycle;
 \node at (7,-3) {,};
 \end{tikzpicture}
 \end{center}
which is shorthand for
\begin{center}
 \begin{tikzpicture}[scale = .2]
 \draw (1,1) -- (1,-3);
 \draw (3,1) -- (3,-3);
 \draw (5,1) -- (5,-3);
\draw [rounded corners] (7,1) -- (7,-1) -- (5,-1);
\draw [fill=black] (5,-1) circle (0.2);
\node at (10,-1) {$+$};
\draw (13,1) -- (13,-3);
 \draw (15,1) -- (15,-3);
 \draw (17,1) -- (17,-3);
 \draw [fill=white, white] (17,-1) circle (0.2);
\draw [rounded corners] (19,1) -- (19,-1) -- (15,-1);
\draw [fill=black] (15,-1) circle (0.2);
\node at (22,-1) {$+$};
\draw (25,1) -- (25,-3);
 \draw (27,1) -- (27,-3);
 \draw (29,1) -- (29,-3);
 \draw [fill=white, white] (29,-1) circle (0.2);
  \draw [fill=white, white] (27,-1) circle (0.2);
\draw [rounded corners] (31,1) -- (31,-1) -- (25,-1);
\draw [fill=black] (25,-1) circle (0.2);
\node at (31,-3) {,};
 \end{tikzpicture}
 \end{center}
 where $\bullet$ represents $[-,-] \in \lie (2)$. 
 
These diagrams are understood to be planar. Hence, for example, the antisymmetry relation for $[-,-]$ is represented by:
\begin{center}
 \begin{tikzpicture}[scale = .2]
 \draw (5,1) -- (5,-3);
\draw [rounded corners] (7,1) -- (7,-1) -- (5,-1);
\draw [fill=black] (5,-1) circle (0.2);
\node at (10,-1) {\quad $=$ \quad $-$};
 \draw (17,1) -- (17,-3);
 \draw [fill=white, white] (17,0) circle (0.2);
\draw [rounded corners] (19,1) -- (19,0) -- (16,0) -- (16, -1) -- (17,-1);
\draw [fill=black] (17,-1) circle (0.2);
\node at (19,-3) {.};
 \end{tikzpicture}
 \end{center}
\end{rem}

 The  family  $\{ \mu(n) \ | \ n \in \nat\}$ satisfies  the following key `centrality' property:
 
 \begin{prop}
 \cite{P_outer}
 For $n,t \in \nat$ and $\phi \in \catlie (n, t)$ (so that $\phi \boxplus 1 \in \catlie (n+1, t+1)$), 
 \[
 \phi \circ \mu (n) = \mu (t) \circ (\phi \boxplus 1) 
 \]
 in $\catlie (n+1, t)$.
 \end{prop} 
 
 \begin{proof}
 This is proved in \cite{P_outer}; for the reader's convenience, a proof is given here. 
 
Since $\catlie (n,t)=0$ if $t=0$ or if $n<t$, we may suppose that $n \geq t >0$. In particular, both $\mu (n)$ and $\mu (t)$ are non-zero. Now, as a PROP, $\catlie$ is generated by $\mu (1)\in \catlie (2,1)$ corresponding to $[-,-] \in \lie (2)$. This allows one to reduce to considering the following cases: 
\begin{enumerate} 
\item 
$n=t$, so that $\phi \in \catlie (n,n) = \rat \sym_n$;
\item
$n=t+1$ and $\phi = \mu (1) \boxplus (t-1)$.
\end{enumerate} 

In the first case, it is straightforward to check the required equivariance $\phi \circ \mu (n) = \mu (n) \circ (\phi \boxplus 1)$. (Indeed, $\mu(n)$ was defined so as to ensure this.)

The second case follows from the Jacobi identity. The generalized form can be denoted diagrammatically (in the case $t=2$):
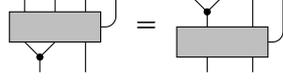
\begin{center}
 \begin{tikzpicture}[scale = .2]
 \draw (1,1) -- (1,-2) -- (2,-3) -- (2,-4);
 \draw (3,1) -- (3,-2)-- (2,-3);
 \draw [fill=black] (2,-3) circle (.2);
 \draw (5,1) -- (5,-4);
\draw [rounded corners] (7,1) -- (7,-1) -- (6,-1);
 \draw [fill= lightgray] (0,0) -- (6,0) -- (6, -2)  -- (0, -2) -- cycle;
 \node at (9,-1) {$=$};
  \draw (12,1) -- (13,0) -- (13,-4);
   \draw (14,1) -- (13,0);
\draw [fill=black] (13,0) circle (.2);
 \draw (16,1) -- (16,-4);
\draw [rounded corners] (18,1) -- (18,-2) -- (17,-2);
 \draw [fill= lightgray] (11,-1) -- (17,-1) -- (17, -3)  -- (11, -3) -- cycle;
 \node at (9,-1) {$=$};
 \end{tikzpicture}
 \end{center}
 where $\bullet$ denotes the Lie bracket. (The reader is encouraged to verify that this follows from the Jacobi identity and that this gives the required identity.)
 \end{proof}
 
 \begin{cor}
 \cite{P_outer}
 \label{cor:mut}
 The family $\{ \mu  (n) \ | \ n \in \nat\}$, induces a natural transformation of endofunctors of $\lmod$:
 \[
 \mut : \delta \rightarrow \id.
 \]
 Evaluated on $n \in \nat$, for a left $\catlie$-module $M$, this is the morphism
 $
\mut(n) : M (n+1) \rightarrow M(n) 
 $ 
  induced by $\mu (n)$.
 \end{cor}
 
\begin{rem}
\label{rem:flie}
In \cite{P_outer}, the category $\lmod$ is denoted $\flie$. Then $\flie^\mu$ is defined to be the full subcategory of objects for which $\mut$ is zero. As is established by \cite{P_outer}, this category is of significant interest, due to its relationship with outer functors on free groups. 
\end{rem}

\begin{exam}
\label{exam:mut_bimodule}
One can apply  $\mut$ in $\catlie$-bimodules. In particular, this gives the natural transformation of $\catlie$-bimodules:
\[
\mut : \delta \catlie \rightarrow \catlie.
\]
 Consider the induced morphism between the associated Schur functors:
\[
\mut : \delta \catlie \otimes_\fb \underline{V} \rightarrow \catlie \otimes _\fb \underline{V}, 
\]
which is a natural transformation of left $\catlie$-modules. This has the form
\[
\mut (V) : \underline{\liealg (V)} \otimes \liealg (V) \rightarrow \underline{\liealg (V)}.
\]
In arity $n$ (with respect to the left $\catlie$-module structure), this identifies as 
$
\liealg (V) ^{\otimes n} \otimes \liealg(V) \rightarrow \liealg(V)^{\otimes n},
$
corresponding to  the $n$th tensor product of the adjoint representation (with action from the right).
\end{exam} 

\begin{rem}
\label{rem:Lie_action}
The fact that the tensor product of the adjoint action is a Lie algebra action reflects a further fundamental property of the morphisms $\mu(n)$, $1 \leq n \in \nat$.

This can be represented diagrammatically as follows. Consider precomposing with the Lie bracket on the distinguished right hand entry (represented by the $\bullet$ in the diagram):

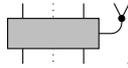
\begin{center}
 \begin{tikzpicture}[scale = .2]
 \draw (1,1) -- (1,-3);
 \draw [dotted] (3,1) -- (3,-3);
 \draw (5,1) -- (5,-3);
 \draw (7,1) -- (7.5,0) -- (8,1);
 \draw [fill=black] (7.5,0) circle (.2); 
\draw [rounded corners] (7.5,0) -- (7.5,-1) -- (6,-1);
 \draw [fill= lightgray] (0,0) -- (6,0) -- (6, -2)  -- (0, -2) -- cycle;
 \node at (8,-3) {.};
 \end{tikzpicture}
 \end{center} 
This is equal to  
 
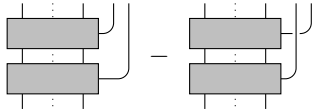
\begin{center}
 \begin{tikzpicture}[scale = .2]
 \draw (1,1) -- (1,-6);
 \draw [dotted] (3,1) -- (3,-6);
 \draw (5,1) -- (5,-6);
\draw [rounded corners] (7,1) -- (7,-1) -- (6,-1);
\draw [rounded corners] (8,1) -- (8,-4) -- (6,-4);
 \draw [fill= lightgray] (0,0) -- (6,0) -- (6, -2)  -- (0, -2) -- cycle;
 \draw [fill= lightgray] (0,-3) -- (6,-3) -- (6, -5)  -- (0, -5) -- cycle;
\node at (10, -2.5) {$-$};
 \draw (13,1) -- (13,-6);
 \draw [dotted] (15,1) -- (15,-6);
 \draw (17,1) -- (17,-6);
 \draw [rounded corners] (20,1) -- (20,-1) -- (18,-1);
 \draw [fill=white, white] (19,-1) circle (.2); 
\draw [rounded corners] (19,1) --  (19,-4) -- (18,-4);
 \draw [fill= lightgray] (12,0) -- (18,0) -- (18, -2)  -- (12, -2) -- cycle;
 \draw [fill= lightgray] (12,-3) -- (18,-3) -- (18, -5)  -- (12, -5) -- cycle;
 \node at (20,-6) {};
 \end{tikzpicture}
 \end{center} 
as can be checked directly from the definition of $\mu (n)$.
 \end{rem}

\section{The adjoint action complex as a DG category}
\label{sect:complex}

The purpose of this section is show that the complex which underlies the calculation of the Lie algebra homology 
 $
H_* (\liealg(V); \liealg(V)^{\otimes n})
$ 
for varying $n$ (naturally in $V$) has additional structure. Namely, it underlies a DG category $\dgcat$. This is proved in Theorem \ref{thm:DG_category}, which puts together the ingredients established in the intervening sections.

In  Section \ref{sect:ce}, $\dgcat$ is related to the universal Chevalley-Eilenberg complex $\ce_\lie$.

\subsection{$\catlie$ as a $\lie$-module}

Using the symmetric monoidal structure $(\rmod , \odot, \rat)$, $\lie$ is a Lie algebra in $\rmod$ (see, for example, \cite[Observation 9.1.3]{MR2494775}).

\begin{rem}
This can be seen by considering the associated Schur functor, as follows. Tautologically, the associated Schur functor  is $V \mapsto \liealg (V)$, which takes values in Lie algebras. Moreover, the structure map $\liealg (\liealg(V)) \rightarrow \liealg (V)$ is a morphism of Lie algebras. Putting these facts together gives the stated property.
\end{rem}

\begin{defn}
\label{defn:lie-modules}
A  $\catlie$-bimodule $B$ is a right $\lie$-module if it is equipped with a structure morphism in $\catlie$-bimodules
$
\alpha : B \odot \lie \rightarrow B
$ 
such that the following equality holds as morphisms of $\catlie$-bimodules $B \odot \lie \odot \lie \rightarrow B$:
\[
\alpha \circ (\id_B \odot [-,-] )
= 
\alpha \circ (\alpha \circ \id_\lie) - \alpha \circ (\alpha \circ \id_\lie) \circ (\id_B \circ \tau )
\]
where $[-,-] : \lie \odot \lie \rightarrow \lie$ is the Lie bracket  and $\tau : \lie \odot \lie \rightarrow \lie\odot \lie$ is given by the symmetry of $\odot$. 

Morphisms between such structures are defined in the obvious way, giving a category. 
\end{defn}

Consider the morphism of $\catlie$-bimodules 
\[
\mut : \delta \catlie \cong \catlie \odot \lie \rightarrow \catlie.
\]

\begin{prop}
\label{prop:mut_Lie_action}
The structure map $\mut : \catlie \odot \lie \rightarrow \catlie$ makes $\catlie$  a right $\lie$-module in $\catlie$-bimodules.
\end{prop}

\begin{proof}
This follows directly from the property of the morphisms $\mu (n)$ presented in Remark \ref{rem:Lie_action}, corresponding to the fact that $\mut$ is constructed from the right adjoint action of $\lie$. 
\end{proof}

Recall that there are inclusions of operads
$ 
\lie \hookrightarrow \ass \hookrightarrow \uass, 
$
where $\ass$ is the associative operad and $\uass$ encodes unital associative algebras; the first morphism encodes the commutator bracket of an associative algebra and the second encodes forgetting the unit. Similarly to the case of $\lie$, $\uass$ is a unital associative algebra in right $\uass$-modules;  it follows by restriction of structure, that $\uass$ is a unital associative algebra in the category of right $\catlie$-modules. 

\begin{rem}
Since $\lie$ is a Lie algebra in right $\catlie$-modules, 
one can form the universal enveloping algebra $U \lie$ of $\lie$ in the category of right $\catlie$-modules, mimicking the classical construction. The inclusion $\lie \hookrightarrow \uass$ induces an isomorphism
\[
U \lie \cong \uass
\]
of unital associative algebras in $\rmod$. (This corresponds to the fact that the universal enveloping $U \liealg (V)$ of a free Lie algebra on $V$ is  naturally isomorphic to $T(V)$, the free unital associative algebra on $V$.)
\end{rem}

Analogously to Definition \ref{defn:lie-modules}, one has:

\begin{defn}
\label{defn:uass-modules}
A $\catlie$-bimodule $B$ is a right  $\uass$-module in $\catlie$-bimodules if it is equipped with a structure map 
 $
\beta : B \odot \uass \rightarrow B 
$ 
in $\catlie$-bimodules that satisfies the unit and associativity constraints. 

Morphisms are defined in the obvious way, giving a category.
\end{defn}

By a straightforward generalization of the classical interpretation of right $\g$-modules (for a Lie algebra $\g$) as right $U\g$-modules, one has:

\begin{prop}
\label{prop:restrict_uass_lie}
Restriction along $\lie \hookrightarrow \uass$ induces an equivalence of categories between the category of right $\uass$-modules in $\bimod$ and the category of right $\lie$-modules in $\bimod$.
\end{prop}

In particular, this applies to the right $\lie$-module structure on $\catlie$ given by $\mut$ (see Proposition \ref{prop:mut_Lie_action}):

\begin{cor}
\label{cor:mut_uass}
The morphism $\mut$ extends canonically to $\mut ^{\uass} : \catlie \odot \uass \rightarrow \catlie$ that makes $\catlie$ a right $\uass$-module in $\bimod$. In particular, the following diagram commutes in $\bimod$:
\[
\xymatrix{
\catlie  \odot \lie 
\ar@{^(->}[d]
\ar[rd]^{\mut}
\\
\catlie \odot \ass
\ar@{^(->}[d]
&
\catlie
\\
\catlie \odot \uass,
\ar[ur]_{\mut^{\uass}}
}
\]
in which the vertical maps are given by applying $\catlie\odot - $ to $\lie \hookrightarrow \ass \hookrightarrow \uass$.

Moreover, by associativity, $\mut^{\uass}$ (and hence $\mut$) is determined by the restriction 
\[
\catlie \odot I \rightarrow \catlie
\]
along the unit morphisms $\xymatrix{ I \ar[r]
\ar@/^1pc/[rr] &\lie \ar@{^(->}[r] & \uass}$.
\end{cor}

\begin{rem}
\label{rem:uass-action_picture}
The property of the  $\uass$-action can be illustrated by the following equality:
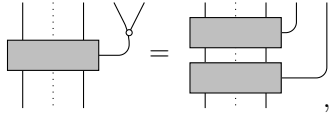
\begin{center}
 \begin{tikzpicture}[scale = .2]
 \draw (1,1) -- (1,-6);
 \draw [dotted] (3,1) -- (3,-6);
 \draw (5,1) -- (5,-6);
\draw [rounded corners] (7,1) -- (8,-1) --(8,-2.5) -- (6,-2.5);
\draw (9,1) -- (8,-1);
\draw [fill=white] (8,-1) circle (.2);
 \draw [fill= lightgray] (0,-1.5) -- (6,-1.5) -- (6, -3.5)  -- (0, -3.5) -- cycle;
\node at (10, -2.5) {$=$};
 \draw (13,1) -- (13,-6);
 \draw [dotted] (15,1) -- (15,-6);
 \draw (17,1) -- (17,-6);
 \draw [rounded corners] (19,1) -- (19,-1) -- (18,-1);
 \draw [fill=white, white] (19.5,0) circle (.2); 
\draw [rounded corners] (21,1) -- (21,-4) -- (18,-4);
 \draw [fill= lightgray] (12,0) -- (18,0) -- (18, -2)  -- (12, -2) -- cycle;
 \draw [fill= lightgray] (12,-3) -- (18,-3) -- (18, -5)  -- (12, -5) -- cycle;
 \node at (21,-6) {,};
 \end{tikzpicture}
 \end{center} 
 in which $\circ$ represents the generator of $\uass(2)$. Via the Corollary, this yields the diagram illustrating Remark \ref{rem:Lie_action}.
 
 This extends in the obvious way by using the associative algebra structure of $\uass$ in right $\catlie$-modules. 
 \end{rem}

\subsection{Restricting to $\catlie \odot I$}
The unit $I \hookrightarrow \lie$ induces a morphism of left $\catlie$, right $\rat\fb$ bimodules:
\[
\catlie \odot I \hookrightarrow \catlie \odot \lie,
\]
where the right $\catlie$-module structure of $\catlie \odot \lie$ has been restricted to a right $\rat\fb$-module structure via $\rat\fb \hookrightarrow \catlie$ induced by the unit $I \hookrightarrow \lie$.

\begin{defn}
\label{defn:delta1}
Let $\delta^{(1)} \catlie \subset \delta \catlie$ be the sub left $\catlie$, right $\rat\fb$ bimodule corresponding to $\catlie \odot I$ under the identification $\delta \catlie \cong \catlie \odot \lie$ of Lemma \ref{lem:delta_catlie}.
\end{defn}

\begin{rem}
By definition, $\delta \catlie (-,-) = \catlie (-, - \boxplus 1)$. In particular, 
\[
\delta \catlie (m,n) = \bigoplus_{f \in \hom_{\fs}(\mathbf{m}, \mathbf{n+1})} \bigotimes_{i=1} ^n \lie (|f^{-1} (i)|).
\]
Under this identification, $\delta^{(1)}\catlie (m, n) $ is the subspace corresponding to   the terms with $f\in \hom_{\fs}(\mathbf{m}, \mathbf{n+1})$ such that $|f^{-1} (n+1)|=1$.
\end{rem}

Recall that there is an isomorphism of the underlying right $\rat\fb$-modules
$ 
\ass \cong \uass \odot I$.  
At the level of the associated Schur functors, writing $\overline{T} (V)$ for the augmentation ideal of $T(V)$, this corresponds to the isomorphism $ \overline{T}(V) \cong T(V) \otimes V$ (this is even an isomorphism of left $T(V)$-modules).

\begin{defn}
\label{defn:pi}
Let $\pi : \catlie \odot \lie \twoheadrightarrow \catlie \odot I$ be the surjective morphism of left $\cat\lie$, right $\rat\fb$ bimodules given by the composite:
\[
 \catlie \odot \lie
\hookrightarrow 
 \catlie \odot \ass 
\cong 
\catlie \odot \uass \odot I
\longrightarrow 
 \catlie \odot I,
\]
where the first map is induced by $\lie \hookrightarrow \ass$ and the second by $\mut^{\uass}$.
\end{defn}

\begin{rem}
The reader should verify that $\pi$ is a morphism of left $\catlie$, right $\rat\fb$ bimodules, as asserted.
\end{rem}

By construction, $\pi$ is a retract of $\catlie \odot I \hookrightarrow \catlie \odot \lie$ in left $\catlie$, right $\rat\fb$ bimodules. Passing to the associated Schur functors, $\pi$ gives
\[
\pi(V) : 
\underline{\liealg(V)} \otimes \liealg(V) 
\twoheadrightarrow 
\underline{\liealg (V)} \otimes V, 
\]
which is a morphism of left $\catlie$-modules. This is determined explicitly by the following:

\begin{lem}
\label{lem:understand_pi}
The natural morphism $\pi(V)$ is determined recursively by the following:
\begin{enumerate}
\item 
for $\mathfrak{Z}\in \underline{\liealg (V)}$ and $x\in V \subset \liealg (V)$, $\pi(V) (\mathfrak{Z} \otimes x) = \mathfrak{Z}\otimes x$;
\item 
for $X, Y \in \liealg (V) $, $\pi(V) (\mathfrak{Z} \otimes [X, Y]) = \pi(V) (\mathfrak{Z}\cdot X \otimes Y) - \pi(V) (\mathfrak{Z}\cdot Y \otimes X)$,  where $\cdot$ denotes the  action of $\liealg(V)$ on $\underline{\liealg(V)}$.
\end{enumerate}
\end{lem}

\begin{proof}
It is straightforward, using the definition of $\pi$, to check that the two equalities in the statement hold.
Moreover, by recursion on the arity of the Lie operad $\lie$ (i.e., on the length of iterated Lie brackets in $\liealg (V)$), these equalities determine $\pi(V)$. 
\end{proof}

\begin{rem}
Via the Schur correspondence, the description of $\pi(V)$ given in Lemma \ref{lem:understand_pi} could be used to {\em define} $\pi$. However, this would requiring checking that this gives a well-defined morphism, i.e., that anti-symmetry and the Jacobi relation are satisfied. Definition \ref{defn:pi} circumvents this, since these relations are already encoded in the action $\mut ^{\uass}$ and $\lie \hookrightarrow \uass$.
\end{rem}

Using Definition \ref{defn:delta1}, $\pi$ can be viewed as a morphism 
\[
\pi : \delta \catlie \twoheadrightarrow \delta^{(1)} \catlie
\]
of left $\catlie$, right $\rat\fb$ bimodules, where $\delta \catlie$ is considered as a left $\catlie$, right $\rat\fb$ bimodule via restriction of its right $\catlie$-module structure along $\rat \fb \hookrightarrow \catlie$.

The key input is the following:

\begin{prop}
\label{prop:delta1_catlie_bimodule}
There is a unique $\catlie$-bimodule structure on $\delta^{(1)} \catlie$ such that $\pi :\delta \catlie \twoheadrightarrow \delta^{(1)} \catlie$ is a morphism of $\catlie$-bimodules.
\end{prop}

\begin{proof}
If such a structure exists, it is clearly unique. Moreover, since $\pi$ is a morphism of left $\catlie$, right $\rat\fb$ bimodules, it suffices to extend the right $\rat \fb$-bimodule structure on $\delta^{(1)} \catlie$ to a suitable right $\catlie$-module structure. 

Since the Schur functor corresponding to $\delta^{(1)} \catlie$ is $V \mapsto \underline{\liealg(V)} \otimes V$, 
we require to exhibit the appropriate natural structure map
\begin{eqnarray}
\label{eqn:structure}
\underline{\liealg(\liealg (V))} \otimes
\liealg (V) 
\rightarrow 
 \underline{\liealg (V)} \otimes V. 
\end{eqnarray}
This should be a morphism of left $\catlie$-modules and also satisfy the unital and associativity constraints for a  right $\catlie$-module structure. Moreover, the compatibility with the morphism $\pi$ requires the 
commutativity of:
\[
\xymatrix{
\underline{\liealg(\liealg (V))} \otimes
\liealg (\liealg(V))
\ar[r]
\ar@{->>}[d]_{\pi(\liealg(V))} 
&
 \underline{\liealg(V)} \otimes
\liealg (V)
\ar@{->>}[d]^{\pi(V)}
\\
\underline{\liealg(\liealg (V))} \otimes
\liealg (V) 
\ar[r]_(.55){(\ref{eqn:structure})}
&
\underline{\liealg (V)} \otimes V,
}
\]
in which the horizontal natural morphisms correspond to the respective right $\catlie$-actions.
Explicitly, the top horizontal map is induced by the Lie algebra composition $\liealg(\liealg(V)) \rightarrow \liealg(V)$   on both tensor factors; we are in the process of constructing the bottom horizontal map, (\ref{eqn:structure}).
 
 Since $\pi(\liealg(V))$ is a retract of the inclusion $\underline{\liealg(\liealg (V))} \otimes
\liealg (V) \hookrightarrow \underline{\liealg(\liealg (V))} \otimes
\liealg (\liealg(V))$ given by $\liealg (V) = I (\liealg (V)) \hookrightarrow \liealg (\liealg (V))$ on the second tensor factor, the commutativity of the square requires that (\ref{eqn:structure}) be defined as the composite
\[
\underline{\liealg(\liealg (V))} \otimes
\liealg (V)
\rightarrow 
\underline{\liealg(V)} \otimes
\liealg (V)
\stackrel{\pi(V)} {\rightarrow}
\underline{\liealg(V)} \otimes V,
\]
where the first map is given by the structure map $\underline{\liealg(\liealg (V))} \rightarrow \underline{\liealg(V)}$ on the first tensor factor (corresponding to the right $\catlie$-module structure of $\catlie$). 

This is clearly a morphism of left $\catlie$-modules, hence it remains to check that it defines a right $\catlie$-module structure on $\delta^{(1)} \catlie$. The unit condition is immediate, hence it suffices to check associativity for the composition.

Since $\catlie$ is a $\catlie$-bimodule (see Example \ref{exam:catlie_bimodule}) and  the adjoint action is an action in right $\catlie$-modules, one reduces to checking the commutativity of the following diagram:
\[
\xymatrix{
\underline{\liealg(V)} \otimes \liealg (\liealg(V))
\ar[rr]^{\lambda_1}
\ar@{^(->}[d]_{\lambda_2}
&&
\underline{\liealg(V)} \otimes \liealg (V)
\ar[dd]^{\pi(V)} 
\\
\underline{\liealg(\liealg(V))} \otimes \liealg (\liealg(V))
\ar[d]_{\pi (\liealg(V))} 
\\
\underline{\liealg(\liealg (V))} \otimes \liealg (V)
\ar[r]_{\lambda_3}
&
\underline{\liealg(V)} \otimes \liealg (V)
\ar[r]_{\pi(V)} 
&
\underline{\liealg(V)} \otimes V,
}
\]
in which $\lambda_1$ and $\lambda_3$ are  induced by the composition $\liealg(\liealg(V)) \rightarrow \liealg (V)$ (on the second and first tensor factors respectively), and $\lambda_2$  
 is given by $\underline{\liealg (-)}$ applied to $V \hookrightarrow \liealg (V)$ on the first tensor factor.
 
Moreover, since $\lie$ is a  binary quadratic operad, one reduces further to considering $\lie_2 (\liealg (V)) \subset \liealg (\liealg(V))$, namely to terms of the form $X\wedge Y$ (considered as an element of $\lie_2 (\liealg(V))$), with $X, Y \in \liealg (V)$; under the map  $\lie_2 (\liealg(V)) \rightarrow \liealg(V)$, $ X \wedge Y \mapsto [X,Y]$.

That the diagram commutes follows from the second equality of Lemma \ref{lem:understand_pi}. Namely, consider $\mathfrak{Z} \in \underline{\liealg (V)}$ and $X$, $Y$ as above. Then, the image of $\mathfrak{Z} \otimes (X \wedge Y)$ on passing around the bottom of the diagram is 
\[
\pi(V) (\mathfrak{Z} \cdot X \otimes Y)  - \pi (V) (\mathfrak{Z} \cdot Y \otimes X).
\]
Passing around the top of the diagram, gives $\pi (V)( \mathfrak{Z} \otimes [X,Y])$. These elements  are equal, by  Lemma \ref{lem:understand_pi}.
\end{proof}

\begin{nota}
\label{nota:mut1}
Write $\mut^{(1)} : \delta^{(1)} \catlie  \rightarrow \catlie $  for the restriction of 
$\mut : \delta \catlie  \rightarrow \catlie $ to $\delta^{(1)} \catlie$.
\end{nota}

The transformations $\mut$ and $\mut^{(1)}$ are also compatible via $\pi$:

\begin{lem}
\label{lem:mut_mut1_compat_pi}
The following diagram commutes:
\[
\xymatrix{
\delta \catlie
\ar[r]^\mut 
\ar@{->>}[d]_\pi 
&
\catlie 
\ar@{=}[d]
\\
\delta^{(1)}\catlie 
\ar[r]_{\mut^{(1)}}
&
\catlie.
}
\]
\end{lem}

\begin{proof}
Consider the composite around the bottom of the diagram, using the identifications $\delta \catlie \cong \catlie \odot \lie$ and $\delta^{(1)} \catlie \cong \catlie \odot I$.
By the construction of $\pi$, this is given by the composite
\[
\xymatrix{
\catlie \odot \lie
\ar[d]
\\
 \catlie \odot \ass 
\ar@{.>}@/^1pc/[drrr]
\ar[d]_\cong 
\\
\catlie \odot \uass \odot I 
\ar[rr]_{\mut^{\uass} \odot I}
&&
\catlie \odot I 
\ar[r]_{\mut^{(1)}}
&
\catlie.
}
\]
By associativity of the right $\uass$ action on $\catlie$, the composite indicated by the dotted arrow is simply the restriction of the action $\catlie \odot \uass \rightarrow \catlie$ to $\catlie \odot \ass$. 

By Corollary \ref{cor:mut_uass},  the map obtained by restriction along $\lie \hookrightarrow \ass$ is  $\mut$, as required.
\end{proof}

From this one deduces: 

\begin{prop}
\label{prop:mut1_bimodule}
The map $\mut^{(1)} : \delta^{(1)} \catlie \rightarrow \catlie$ is a morphism of $\catlie$-bimodules.
\end{prop}

\begin{proof}
It is clear that $\mut^{(1)}$ is a morphism of left $\catlie$-modules, hence it suffices to check that it is a morphism of right $\catlie$-modules. 

Consider the following diagram
\[
\xymatrix{
(\catlie\odot I) \otimes_\fb \catlie 
\ar[rr]^{\mut^{(1)} \otimes \id}
\ar[d]
&&
\catlie \otimes _\fb \catlie 
\ar[d]
\\
\catlie\odot \lie
\ar[rr]^\mut
\ar[d]_{\pi} 
&&
\catlie
\ar@{=}[d]
\\
\catlie \odot I 
\ar[rr]_{\mut^{(1)}}
&&
\catlie.
}
\]
Here, the vertical arrows in the top square are given respectively by the right $\catlie$ action on  $\catlie$ and the restriction of the right $\catlie$ action on $\catlie \odot \lie$ to $\catlie \odot I$. The top square commutes since $\mut$ is a morphism of right $\catlie$-modules. The bottom square is commutative by Lemma \ref{lem:mut_mut1_compat_pi}, so that the whole diagram is commutative. 

Now, by construction (see the proof of Proposition \ref{prop:delta1_catlie_bimodule}), the left hand vertical composite is the right $\catlie$-module structure map of $\catlie \odot I$. The commutativity of the outer square therefore gives that $\mut^{(1)}$ is a morphism of right $\catlie$-modules, as required.
\end{proof}

\subsection{The DG category}

To motivate the construction in this section, we first consider the following simpler situation. Let  $A$ be a unital, associative algebra and $M$  an $A$-bimodule, both concentrated in (homological) degree $0$. Write $M[1]$ for the shift of $M$, placed in homological degree one. Then one can form the associated square-zero extension:
\[
A \oplus M[1],
\]
which is a graded associative algebra concentrated in homological degrees zero and one. 

Suppose that $d : M \rightarrow A$ is a morphism of $A$-bimodules. This induces a degree $-1$ morphism $M[1] \rightarrow A$, which can be considered as a differential. This extends  to a differential on  $(A \oplus M[1])^{\otimes 2}$ (considered as a tensor product of complexes). In particular, restricted to $M[1] \otimes M[1]$, this gives
\[
M[1] \otimes M[1] \rightarrow A \otimes M[1]  \  \oplus \   M[1] \otimes A . 
\]
Explicitly, for $x, y \in M$, $x[1] \otimes y[1] \mapsto dx \otimes y[1] - x[1] \otimes dy$.

\begin{lem}
\label{lem:DGA_equivalent}
The morphism of $A$-bimodules $d : M \rightarrow A$ induces a differential graded algebra structure on $(A \oplus M[1] , d)$ if and only if the following equivalent conditions are satisfied:
\begin{enumerate}
\item 
the composite 
\[
M[1] \otimes M[1] \stackrel{d}{\rightarrow} A \otimes M[1]  \  \oplus \   M[1] \otimes A
\rightarrow M[1]
\]
is zero, where the second map is induced by the left and right $A$-actions on $M$ respectively; 
\item 
the following diagram commutes
\[
\xymatrix{
M \otimes M \ar[r]^{d \otimes \id_M}
\ar[d]_{\id_M \otimes d}
&
A \otimes M \ar[d] 
\\
 M \otimes A \ar[r] 
 &M,
}
\]
where the unlabelled arrows are given by the left and right $A$-module structures of $M$ respectively.
\end{enumerate}
\end{lem}

\begin{proof}
One has a DGA if and only if the multiplication $(A\oplus M[1]) ^{\otimes 2} \rightarrow A\oplus M[1]$ is compatible with the differentials. This multiplication has components:
\begin{eqnarray*}
A \otimes A & \rightarrow & A
\\
A \otimes M[1] \ \oplus \  M[1] \otimes A &\rightarrow & M[1] 
\\
M[1] \otimes M[1] & \rightarrow & 0,
\end{eqnarray*}
corresponding to the multiplication of $A$, the $A$-bimodule structure maps of $M[1]$ and the `square zero' condition (which also follows here from the homological grading).  

Since $d: M \rightarrow A$ is a morphism of $A$-bimodules, the compatibility for the first two components is immediate. The first condition of the statement is equivalent to the compatibility with the final component, giving the first equivalent condition. 

The second equivalent condition is simply a restatement without the homological grading, taking into account the Koszul signs.
\end{proof}

\begin{rem}
\label{rem:generators}
Both the composite maps $M \otimes M \rightrightarrows M$ considered in the second equivalent condition of Lemma \ref{lem:DGA_equivalent}  are morphisms of $A$-bimodules, where $A$ acts on the left via the left hand tensor factor of $M \otimes M$ and on the right via the right hand tensor factor. Hence, to check the condition, one can restrict to considering  generators of $M \otimes M$ as an $A$-bimodule.
\end{rem}

Here we adapt the above, working with unital monoids in the category of $\rat\fb$-bimodules and bimodules over such, using the monoidal structure provided by  $\otimes_\fb$. We take $A$ to be $\catlie$ and $M$ to be $\delta^{(1)} \catlie$, with $d = \mut ^{(1)}$; $\catlie$ is a unital monoid in $\rat\fb$-bimodules  and $\delta^{(1)} \catlie$ is a bimodule over this. 

Proposition \ref{prop:mut1_bimodule} tells us that $d$ is a morphism of $\catlie$-bimodules. We may therefore consider the $\rat\fb$-bimodule 
\[
\catlie \oplus \delta^{(1)}\catlie[1]
\]
as a category enriched in $\nat$-graded $\rat$-vector spaces. As for $\catlie$, the set of  objects is $\nat$ and composition of morphisms is given by the square zero extension, which yields
\[
(\catlie \oplus \delta^{(1)}\catlie[1])
\otimes_\fb
(\catlie \oplus \delta^{(1)}\catlie[1])
\rightarrow 
\catlie \oplus \delta^{(1)}\catlie[1].
\]

For the proof of Theorem \ref{thm:DG_category} below, we require suitable generators for $\delta^{(1)}\catlie$. 
Recall that, for $a, b \in \nat$, $\delta^{(1)} \catlie (a,b) \subset \catlie (a, b \boxplus 1) = \catlie (a, b+1)$. 

\begin{nota}
\label{nota:iota}
For $0<a \in \nat$, let $\iota_a\in \delta^{(1)} \catlie(a,a-1)$ be the element corresponding to $\id_a \in \catlie (a,a)$.
(By convention, $\iota_0=0$.)
\end{nota}

This family provides generators for $\delta^{(1)} \catlie$ in the following sense:

\begin{lem}
\label{lem:iota_generate}
The family $\{ \iota_a \ | \ a \in \nat \}$ generates $\delta^{(1)} \catlie $ in both of the following ways:
\begin{enumerate}
\item 
as a right $\catlie$-module; 
\item 
as a left $\catlie$, right $\rat\fb$ bimodule. 
\end{enumerate}
\end{lem}

\begin{proof}
That the elements generate $\delta^{(1)}\catlie$ as a right $\catlie$-module is clear.

The left $\catlie$-module generated by these elements when evaluated on $(a,b)$ gives the subspace  
 $$\bigoplus_f \bigotimes_{i=1}^{b+1} \lie( |f^{-1} (i)|)$$
 of $\delta \catlie (a,b)$, where $f$ ranges over the surjective set maps $f : \mathbf{a} \twoheadrightarrow \mathbf{b+1}$ such that $f^{-1} (b+1) =\{ a \}$. Once the right action by $\sym_a$ is taken into account, one obtains all of $\delta^{(1)} \catlie (a,b)$. 
\end{proof}

\begin{thm}
\label{thm:DG_category}
The differential $\mut^{(1)} : \delta^{(1)} \catlie[1] \rightarrow \catlie$ makes $(\catlie \oplus \delta^{(1)}\catlie [1], \mut^{(1)})$ into a differential graded category.

The inclusion $(\catlie , 0) \hookrightarrow ( \catlie \oplus \delta^{(1)}\catlie[1], \mut^{(1)})$ of the subcomplex in homological degree $0$ is a functor of DG categories.
\end{thm}

\begin{proof}
By Lemma \ref{lem:DGA_equivalent} (adapted to the current setting), it suffices to show that, for all natural numbers $m,n,t \in \nat$, the following diagram commutes:
\[
\xymatrix{
\delta^{(1)} \catlie (n,t) \otimes _{\sym_n} \delta^{(1)} \catlie (m,n) 
\ar[rr]^{\id \otimes \mut^{(1)}}
\ar[d]_{\mut^{(1)} \otimes \id}
&&
\delta^{(1)} \catlie (n,t) \otimes_{\sym_n} \catlie (m,n) 
\ar[d]
\\
\catlie (n,t) \otimes _{\sym_n} \delta^{(1)} \catlie (m,n) 
\ar[rr]
&&
\delta^{(1)}\catlie (m,t),
}
\]
in which the unlabelled maps are given by the left and right $\catlie$-module structures of $\delta^{(1)}\catlie$ respectively.

Using Notation \ref{nota:iota}, one has the elements $\iota_n \in \delta^{(1)} \catlie (n, n-1)$  and $\iota_{n+1} \in \delta^{(1)} \catlie (n+1, n)$ (we may assume $n \geq 1$ here). Then, by using Lemma \ref{lem:iota_generate} and 
arguing as in Remark \ref{rem:generators}, one reduces to studying the images of the following family:
 $$
 \iota_n \otimes \iota_{n+1} \in \delta^{(1)} \catlie (n,n-1) \otimes _{\sym_n} \delta^{(1)} \catlie (n+1,n)
\quad
1 \leq n \in \nat 
 .
 $$ 
For each $n$, one considers the two images of $\iota_n \otimes \iota_{n+1}$ in $\delta^{(1)}\catlie (n+1,n-1)$ given by the two paths around the above square. We require to show that these are equal.

By the definition of $\mut$ and of the elements $\mu (*)$, one has 
\begin{eqnarray*}
\mut^{(1)} (\iota_n)&=&\mu (n-1)\in \catlie(n,n-1) ; 
\\
\mut^{(1)} (\iota_{n+1})&= &\mu (n) \in \catlie (n+1,n).
\end{eqnarray*}

Hence, passage around the bottom of the diagram gives $\mu(n-1) \boxplus 1$, considered as an element of $\delta^{(1)}\catlie (n+1,n-1)  \subset  \catlie (n+1,(n-1) \boxplus 1)= \catlie(n+1,n)$. This is represented by the diagram 

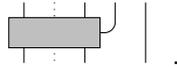
\begin{center}
 \begin{tikzpicture}[scale = .2]
 \draw (1,1) -- (1,-3);
 \draw [dotted] (3,1) -- (3,-3);
 \draw (5,1) -- (5,-3);
 \draw (9,1) -- (9,-3);
\draw [rounded corners] (7,1) -- (7,-1) -- (6,-1);
 \draw [fill= lightgray] (0,0) -- (6,0) -- (6, -2)  -- (0, -2) -- cycle;
 \node at (11,-3) {.};
 \end{tikzpicture}
 \end{center}
 
To pass around the top of the diagram, we have to use the right $\catlie$-module structure of $\delta^{(1)} \catlie$. For this, we write $\mu (n)$ as $\mu'(n) + \mu''(n)$, given by 

 \begin{center}
 \begin{tikzpicture}[scale = .2]
\node at (-5,-1) {$\mu'(n) :=$};
 \draw (1,1) -- (1,-3);
 \draw [dotted] (3,1) -- (3,-3);
 \draw (5,1) -- (5,-3);
 \draw (7,1) -- (7,-3);
 \draw [fill=white, white] (7,-1) circle (0.2);
\draw [rounded corners] (9,1) -- (9,-1) -- (6,-1);
 \draw [fill= lightgray] (0,0) -- (6,0) -- (6, -2)  -- (0, -2) -- cycle;
 \end{tikzpicture}
 \end{center} 
 
 \begin{center}
 \begin{tikzpicture}[scale = .2]
 \node at (8,-1) {$\mu''(n) :=$};
\draw (14,1) -- (14,-3) ;
 \draw [dotted] (16,1) -- (16,-3);
 \draw (18,1) -- (18,-3) ;
 \draw (20,1) -- (20,-3) ;
\draw  [rounded corners] (22,1) -- (22,-1) -- (20,-1); 
\draw [fill= black] (20,-1) circle (0.2);
\node at (22,-3) {.};
 \end{tikzpicture}
 \end{center}
(Note that $\mu'(1)=0$.) 
 
The action of $\mu'(n)$ on the right on $\iota_{n-1}$ does not involve $\pi$, whereas the action of $\mu''(n)$ does.  
  Explicitly, by the definition of the right $\catlie$-module structure of $\delta^{(1)}\catlie$,  the action of $\mu''(n)$ on  $\iota_{n-1}$ gives the following element (considered as belonging to $\delta^{(1)} \catlie (n+1, n-1)$):
 \begin{center}
 \begin{tikzpicture}[scale = .2]
 \draw (1,1) -- (1,-3);
 \draw [dotted] (3,1) -- (3,-3);
 \draw (5,1) -- (5,-3);
 \draw (9,1) -- (9,-3);
\draw [rounded corners] (7,1) -- (7,-1) -- (6,-1);
 \draw [fill= lightgray] (0,0) -- (6,0) -- (6, -2)  -- (0, -2) -- cycle;
 \node at (11,-1) {-};
 
  \draw (14,1) -- (14,-3);
 \draw [dotted] (16,1) -- (16,-3);
 \draw (18,1) -- (18,-3);
 \draw (20,1) -- (20,-3);
 \draw [fill=white, white] (20,-1) circle (0.2);
\draw [rounded corners] (22,1) -- (22,-1) -- (19,-1);
 \draw [fill= lightgray] (13,0) -- (19,0) -- (19, -2)  -- (13, -2) -- cycle;
 \node at (22,-3) {.};
 \end{tikzpicture}
 \end{center}

Summing with the contribution from the action of  $\mu' (n)$ thus gives the element written as $\mu (n-1) \boxplus 1$ above. This establishes the commutativity of the diagram and hence the result.
\end{proof}

\begin{nota}
\label{nota:dgcat}
Write $\dgcat$ for the DG category $ (\catlie \oplus \delta^{(1)}\catlie [1], \mut^{(1)})$, equipped with the inclusion 
$
\catlie \hookrightarrow \dgcat
$ 
of DG categories, where $\catlie$ is considered as a DG category with morphisms concentrated in homological degree zero.
\end{nota}

 \section{The Chevalley-Eilenberg complex}
 \label{sect:ce}

 The purpose of this section is to explain the relationship between the complex of $\catlie$-bimodules underlying $\dgcat$ and the universal Chevalley-Eilenberg complex related to calculating the Lie algebra homology $H_* (\g; \underline{\g})$, for any $\rat$-Lie algebra $\g$. 
 
 \subsection{Constructing the complex}
 
 Recall that, for a Lie algebra $\g$ and a right $\g$-module $M$, the Chevalley-Eilenberg complex $\ce(\g; M)$ has the form $(M \otimes \Lambda^* \g, d)$, where $M\otimes \Lambda^t \g$ is in homological degree $t$. The tail of the complex is thus 
 \[
 \ldots 
 \rightarrow 
 M \otimes \Lambda^2 \g
 \rightarrow 
 M \otimes \g 
 \rightarrow 
 M
 \]
 with differentials given for $m \in M$ and $x,y\in \g$ by:
 \begin{eqnarray*}
 m \otimes (x \wedge y) & \mapsto & m\cdot x \otimes y  - m \cdot y \otimes x - m \otimes [x,y] \\
 m \otimes x & \mapsto & m\cdot x,
 \end{eqnarray*}
where $\cdot$ denotes the action of $\g$ on $M$. The homology of $\ce(\g; M)$ is the Lie algebra homology $H_* (\g; M)$ of $\g$ with coefficients in the module $M$.

In particular, this can be applied taking $M = \g^{\otimes n}$, the $n$-fold tensor product of the (right) adjoint representation. These assemble to $\underline{\g}$ using the universal example $\mut$, via Proposition \ref{prop:mut_Lie_action}, which gives that 
\[
\mut : \catlie \odot \lie \rightarrow \catlie 
\]
is a morphism of $\catlie$-bimodules which makes $\catlie$ a right $\lie$-module, in the sense of Definition \ref{defn:lie-modules}.

Using the tensor product $- \otimes_{\catlie}\underline{\g}$, $\mut$ gives the morphism of left $\catlie$-modules 
$
\underline{\g} \otimes \g \rightarrow \underline{\g},
$
 where the action is given by that on $\underline{\g}$. In arity $n$, this coincides with the $n$-fold tensor product of the right adjoint representation $\g^{\otimes n} \otimes \g \rightarrow \g^{\otimes n}$, by the construction of $\mut$. 

This gives the following:

\begin{prop}
\label{prop:CE_g}
For a Lie algebra $\g$, $\ce (\g; \underline{\g})$ is naturally a complex of left $\catlie$-modules. Hence the Lie algebra homology groups $H_*(\g; \underline{\g})$ take values naturally in left $\catlie$-modules.  
\end{prop}

As already suggested by the usage of $\mut$ to give the left $\catlie$-module structure in Proposition \ref{prop:CE_g}, it is natural to consider the universal example. This corresponds to working in the symmetric monoidal category $(\rmod, \odot , \rat)$ rather than $(\qmod, \otimes, \rat)$, using the fact that $\lie$ is a Lie algebra in $\rmod$. 

\begin{nota}
Write $\ce_\lie$ for the universal Chevalley-Eilenberg complex that encodes $\g \mapsto \ce(\g; \underline{\g})$.  
\end{nota}

\begin{rem}
This complex $\ce_\lie$ can be constructed using the Schur correspondence as follows. The associated complex of Schur functors $V \mapsto   \ce_\lie (V)$ is 
\[
\ce_\lie (V) = \ce( \liealg (V); \underline{\liealg(V)} ); 
\]
the right $\catlie$-module structure is induced by the composition morphism $\liealg (\liealg (V)) \rightarrow \liealg (V)$ corresponding to the operadic composition $\lie \circ \lie \rightarrow \lie$. 
\end{rem}

The key properties are resumed in the following:

\begin{prop}
\label{prop:CE_universal}
\ 
\begin{enumerate}
\item 
The complex $\ce_\lie$ is a complex of $\catlie$-bimodules. 
\item 
Restricted to the right $\catlie$-module structure, $\ce_\lie$ is a complex of projective right $\catlie$-modules.
\item
For a Lie algebra $\g$, there is a natural isomorphism of complexes of left $\catlie$-modules
\[
\ce( \g; \underline{\g}) \cong \ce_\lie \otimes_{\catlie} \underline{\g}.
\]
\end{enumerate}
\end{prop}

\begin{proof}
The first statement is an immediate generalization of Proposition \ref{prop:CE_g}. Indeed, using the Schur correspondence and the construction of $\ce_\lie$ outlined above, it is a Corollary of Proposition \ref{prop:CE_g}.

The second statement follows from the observation that, for a Lie algebra $\g$, the $t$th term $\g^{\otimes n} \otimes \Lambda^t \g$ of $\ce(\g; \g^{\otimes n}) $ is naturally (with respect to $\g$) a direct summand of $\g^{\otimes n+t}$, since we are working over $\rat$. This implies that, for the universal complex, the $t$th term of $\big( \ce_\lie \big) (-, n)$ is a direct summand of $\catlie (-, n+t)$ as a right $\catlie$-module. (Here $n$ is the degree corresponding to the left $\catlie$-module structure.) In particular, this implies that the $t$th term of $\ce_\lie$ is projective as a right $\catlie$-module.

The final statement essentially holds by construction. It generalizes the tautological fact that the Lie algebra structure of $\lie$ in right $\catlie$-modules induces an isomorphism of Lie algebras 
$
\g \cong \lie \otimes_{\catlie} \underline{\g}$.
\end{proof}

\begin{rem}
\label{rem:gen_CE_homology}
\ 
\begin{enumerate}
\item 
More generally, one can consider the functor 
\[
\ce_\lie \otimes _{\catlie} - \ : \ \lmod \rightarrow \mathrm{Ch}(\lmod)
\]
from left $\catlie$-modules to chain complexes of left $\catlie$-modules, together with the associated homology with values in left $\catlie$-modules:
\[
M \mapsto H_* (\ce_\lie \otimes_{\catlie} M). 
\]
Since $\ce_\lie$ has terms that are projective as right $\catlie$-modules, this is a homological functor on $\lmod$.
 When evaluated with $M = \underline{\g}$, for a Lie algebra $\g$, one recovers $H_* (\g; \underline{\g})$.
 \item 
 \label{item:CE_hom_left_derived_functors} 
The functor  $ M \mapsto H_* (\ce_\lie \otimes_{\catlie} M)$ is related to the left derived functors of $M \mapsto H_0
(\ce_\lie \otimes_{\catlie} M)$. However, they do not coincide, since $H_1$ does not in general vanish when $M$ is a projective left $\catlie$-module, as is seen by considering the case of free Lie algebras.
\end{enumerate}
\end{rem}

\subsection{Relating $\ce_\lie$ and $\dgcat$ as complexes of $\catlie$-bimodules}

In this subsection, $\dgcat$ is considered only as a complex of $\catlie$-bimodules. 

\begin{rem}
\label{rem:subcomplex}
As a complex of left $\catlie$-modules, $\dgcat$ was essentially constructed as a subcomplex of $\ce_\lie$ (cf. Notation \ref{nota:mut1}). Note, however, that this is not a subcomplex as a $\catlie$-bimodule.
\end{rem}

\begin{thm}
\label{thm:ce_dgcat}
The morphism $\pi :\delta \catlie \twoheadrightarrow \delta^{(1)} \catlie$ of $\catlie$-bimodules induces a morphism of complexes of $\catlie$-bimodules 
\[
\ce_\lie \twoheadrightarrow \dgcat
\]
which is a weak equivalence. 

This is  determined by the commutative diagram in $\catlie$-bimodules:
\[
\xymatrix{
(\ce_\lie)_2 
\ar[r]
\ar[d]
&
(\ce_\lie)_1 = \delta \catlie 
\ar@{->>}[d]_\pi 
\ar[r]^{\mut}
&
(\ce_\lie)_0 = \catlie 
\ar@{=}[d]
\\
0 
\ar[r]
&
\delta^{(1)}\catlie
\ar[r]_{\mut^{(1)}}
 &
 \catlie.
}
\]
\end{thm}

\begin{proof}
The right hand square in the diagram is the commutative diagram given by Lemma \ref{lem:mut_mut1_compat_pi}. (Recall that $\pi$ and $\mut^{(1)}$  are morphisms of $\catlie$-bimodules by Propositions \ref{prop:delta1_catlie_bimodule} and \ref{prop:mut1_bimodule} respectively; that $\mut$ is a morphism of $\catlie$-bimodules is explained in Example \ref{exam:mut_bimodule}.)

It follows that the existence of the required chain complex map is equivalent to the vanishing of the composite $(\ce_\lie)_2 \rightarrow \delta \catlie \stackrel{\pi}{\rightarrow} \delta^{(1)}\catlie$, as asserted in the diagram. This vanishing can be established using the Schur correspondence, which reduces to considering the composite
\[
\underline{\liealg (V)} \otimes \Lambda ^2 (\liealg (V))
\rightarrow 
\underline{\liealg (V)} \otimes \liealg (V)
\stackrel{\pi}{\rightarrow}
\underline{\liealg (V)} \otimes V
\]
in which the first map is the Chevalley-Eilenberg differential. 

Since these correspond to morphisms of right $\catlie$-modules, one can reduce using the inclusion $\Lambda^2 (V) \hookrightarrow \Lambda ^2 (\liealg (V))$ induced by $V \subset \liealg (V)$ to considering the restriction:
\[
\underline{\liealg (V)} \otimes \Lambda ^2 (V)
\rightarrow 
\underline{\liealg (V)} \otimes \liealg (V)
\stackrel{\pi(V)}{\rightarrow}
\underline{\liealg (V)} \otimes V.
\]

Consider $\mathfrak{Z} \in \underline{\liealg (V)}$ and $x,y \in V$. The image of $\mathfrak{Z} \otimes x \wedge y \in \underline{\liealg (V)} \otimes \Lambda ^2 (V)$ under the Chevalley-Eilenberg differential is $\mathfrak{Z}\cdot x \otimes y 
 - \mathfrak{Z} \cdot y \otimes x - \mathfrak{Z} \otimes [x,y]$ in $\underline{\liealg (V)} \otimes \liealg (V)$. Lemma \ref{lem:understand_pi} shows that this element lies in the kernel of $\pi(V)$. This establishes the required vanishing. 
 
By construction, this gives the surjection of complexes of $\catlie$-bimodules $\ce_\lie \twoheadrightarrow \dgcat$. Moreover, it is clear that this is a retract of the inclusion (as complexes of left $\catlie$-modules) recalled in Remark \ref{rem:subcomplex}. This inclusion induces an isomorphism in homology by standard results of homological algebra, hence the result follows from two out of three.
\end{proof}

\begin{cor}
\label{cor:nat_trans_ce_dgcat_homol}
For $M \in \ob \lmod$, there is a natural surjection of complexes of left $\catlie$-modules
\[
\ce_\lie \otimes_{\catlie} M \twoheadrightarrow \dgcat \otimes_\catlie M, 
\]
induced by the surjection $\ce_\lie \twoheadrightarrow \dgcat$ of Theorem \ref{thm:ce_dgcat}.

Hence one obtains the natural transformation $H_* (\ce_\lie \otimes_{\catlie} M )\twoheadrightarrow H_*(\dgcat \otimes_\catlie M)$ of (graded) left $\catlie$-modules.  This is an isomorphism in homological degree zero and a surjection in degree one. 
\end{cor}

\begin{proof}
The first statement follows from Theorem \ref{thm:ce_dgcat}, using the right exactness of $- \otimes_\catlie M$. 

For the second statement, one uses the long exact sequence associated to the short exact sequence of complexes of left $\catlie$-modules
\[
0
\rightarrow 
\ker \rightarrow 
\ce_\lie \otimes_{\catlie} M \rightarrow \dgcat \otimes_\catlie M
\rightarrow 
0.
\]
Since $\ker$ is zero in homological degree zero, the result follows.
\end{proof}

\begin{rem}
Whilst $\ce_\lie$ has the significant advantage over $\dgcat$ that its terms are projective as right $\catlie$-modules, we have not taken into the key additional property of $\dgcat$ that leads to its structure as a DG category (see Theorem \ref{thm:DG_category}). Moreover, the underlying complex of $\dgcat$ has the advantage of being supported in homological degrees zero and one.
\end{rem}

\subsection{An example}
\label{subsect:exam}

By Corollary \ref{cor:nat_trans_ce_dgcat_homol}, one has the natural isomorphism
\[
H_0 (\ce_\lie \otimes_{\catlie} M )\stackrel{\cong}{\rightarrow} H_0(\dgcat \otimes_\catlie M)
\]
 of  left $\catlie$-modules, for $M \in \ob \lmod$. Hence, as in Remark \ref{rem:gen_CE_homology} (\ref{item:CE_hom_left_derived_functors}), the homology $H_* (\ce_\lie \otimes_{\catlie} M )$ gives an approximation to the left derived functors of $M \mapsto H_0(\dgcat \otimes_\catlie M)$. The latter functor is of interest in relation to the theory of \cite{P_outer}, since it identifies with the functor of Proposition \ref{prop:induction} below. 
 
\begin{rem} 
Recall that there is an augmentation of operads $\lie \rightarrow I$ and this induces the augmentation $\catlie \rightarrow \rat \fb$. Hence, any left $\rat\fb$-module can be considered as a left $\catlie$-module by restriction along the augmentation. 
\end{rem} 

From the point of view of the theory of \cite{P_outer}, fundamental examples are given by $M = \rat \sym_n$, for $n \in \nat$, where $M$ is considered as a left $\catlie$-module supported on $\mathbf{n}$. The left derived functors of $M \mapsto H_0(\dgcat \otimes_\catlie M)$ in this case are of major interest: they reflect the different cohomological behaviour between functors on free groups and {\em outer} functors on free groups.

As a first approximation, one considers the homology:
\[
H_* (\ce_\lie \otimes_{\catlie} \rat \sym_n ).
\]
By construction, this is $\nat$-graded and takes values in left $\catlie$, right $\rat \sym_n$ bimodules, where the right $\rat\sym_n$ structure is induced by the right regular action on $\rat \sym_n$.

\begin{prop}
\label{prop:H_ce_QSn}
For $n \in \nat$, the left $\catlie$-module structure on $H_* (\ce_\lie \otimes_{\catlie} \rat \sym_n) $ is induced by the underlying left $\rat\fb$-module structure.

The underlying left $\rat\fb$, right $\rat \sym_n$ bimodule structure is given by
\[
H_t (\ce_\lie \otimes_{\catlie} \rat \sym_n )(\mathbf{m}) = 
\left\{
\begin{array}{ll}
\rat \sym_m \odot \mathrm{sgn_t} & m+t = n \\
0 & \mbox{otherwise,}
\end{array}
\right.
\]
where the convolution product $\odot$  is formed with respect to the respective right actions of $\sym_m$ and $\sym_t$, giving a right $\rat \sym_n$-module, and $\rat \sym_m$ acts on the left via the left regular action on $\rat \sym_m$. 
\end{prop}

\begin{proof}
This follows directly from the identification of the complex $\ce_\lie \otimes_{\catlie} \rat \sym_n$. As in the proof of Proposition \ref{prop:CE_universal}, evaluated on $\mathbf{m}$ and in homological degree $t$, this is a direct summand of 
\[
\catlie (-, m+t) \otimes_{\catlie} \rat \sym_n.
\] 
By the Yoneda lemma, the latter is non-zero if and only if $m+t=n$, in which case one recovers $\rat \sym_n$. 

It follows easily that the differential on $\ce_\lie \otimes_{\catlie} \rat \sym_n$ is zero. Hence it remains to identify the appropriate direct summand. In terms of the Chevalley-Eilenberg complex for a Lie algebra $\g$, this corresponds to $\g^{\otimes m} \otimes \Lambda^t (\g) \hookrightarrow \g^{\otimes m} \otimes \g^{\otimes t}$, induced by $\Lambda^t (\g) \subset \g^{\otimes t}$. Unravelling the identifications, one obtains the stated result. 
\end{proof}

\section{The homology of $\dgcat$ and applications}
\label{sect:consequences}

In this section, we consider some consequences of the DG category structure on $\dgcat$ and its relationship with $\catlie$ via 
the functor $
\catlie \rightarrow \dgcat
.
$

\subsection{Passing to homology}
Having a DG category in hand, one naturally considers its homology. Generalizing the fact that the homology of a DG algebra is a graded associative algebra, the homology has the structure of a category enriched in $\nat$-graded $\rat$-vector spaces.

\begin{nota}
\label{nota:homology}
Write $\homol$ for the homology of $\dgcat = (\catlie \oplus \delta^{(1)}\catlie [1], \mut^{(1)})$ in $\fb$-bimodules and 
$\homol_0$ (respectively $\homol_1$) for the degree zero (resp. one) components.
\end{nota}

Theorem \ref{thm:DG_category} implies immediately:

\begin{prop}
\label{prop:homol}
The following properties hold:
\begin{enumerate}
\item 
$\homol$ is a small category enriched in $\nat$-graded $\rat$-vector spaces concentrated in homological degrees zero and one, with set of objects $\nat$; 
\item 
$\homol_0$ is a small $\rat$-linear category with set of objects $\nat$;  
\item 
there is a full, $\rat$-linear functor 
$ 
\catlie \rightarrow \homol_0
$ 
that is the identity on objects; in particular,  $\homol_0$ has an underlying $\catlie$-bimodule structure; 
\item  
$\homol_1$ is an $\homol_0$-bimodule and hence has  an underlying $\catlie$-bimodule structure.
\end{enumerate}
\end{prop}

Using the restricted right $\rat \fb$-module structures on $\homol_0$, $\homol_1$, one can form the associated Schur functors 
\begin{eqnarray*}
\homol_0 \otimes_\fb \underline{V} & \mbox{\ and \ } & 
\homol_1 \otimes_\fb \underline{V};
\end{eqnarray*}
these take values in left $\homol_0$-modules (and hence in left $\catlie$-modules).

The underlying left $\rat\fb$-modules identify as follows:

\begin{prop}
\label{prop:relate_Lie_algebra_homology}
For $n \in  \nat$, there are natural isomorphisms for $\epsilon \in \{0, 1\}$:
\begin{eqnarray*}
\big( \homol_\epsilon \otimes_\fb \underline{V}\big)(n)  & \cong & H_\epsilon (\liealg (V); \liealg(V)^{\otimes n}),
\end{eqnarray*}
where the right hand side denotes Lie algebra homology of the free Lie algebra $\liealg(V)$ with coefficients in the $n$-fold tensor product of the adjoint representation.
\end{prop}

\begin{proof}
By definition, the complex underlying the DG category $\dgcat$  is 
 $
\catlie \odot I \rightarrow \catlie
$ 
with differential induced by the tensor products of the adjoint representation. On applying the exact functor $- \otimes_\fb \underline{V}$, one obtains 
\[
\underline{\liealg (V)} \otimes V \rightarrow \underline{\liealg (V)}
\]
which, evaluated in arity $n$ gives
$ 
\liealg (V)^{\otimes n} \otimes V \rightarrow \liealg (V)^{\otimes n}
$, 
with differential the restriction of the $n$-fold tensor product of the adjoint action of $\liealg (V)$ from the right. 
Since the homology of this complex is the stated Lie algebra homology,  the result follows.
\end{proof}

\begin{cor}
\label{cor:lie_alg_homol_H0_modules}
The functors $H_\epsilon (\liealg (V); \liealg(V)^{\otimes *})$, for $\epsilon \in \{0 ,1\}$, take values in left $\homol_0$-modules.
\end{cor}
\subsection{$\homol_0$-modules}

Recall from Remark \ref{rem:flie} that, by definition, $\flie^\mu$ is the full subcategory of $\lmod$ of objects for which the natural transformation  $\mut$ is zero (see \cite{P_outer} for more details). The category $\flie^\mu$ models outer functors on free groups.

\begin{thm}
\label{thm:fliemu_homol0}
The category $\flie^\mu$ is equivalent to the category of left $\homol_0$-modules and, with respect to this equivalence,  the canonical inclusion $\flie^\mu \hookrightarrow \flie$ is induced by restriction along the full, $\rat$-linear functor $\catlie \rightarrow \homol_0$.
\end{thm}

\begin{proof}
By the results of \cite{P_outer}, $\flie^\mu$ is abelian with enough projectives. Moreover, for $n \in \nat$,   $\homol_0(n,- )$ (considered as a left $\catlie$-module) identifies as the projective in $\flie^\mu$ that corepresents evaluation on $n$.
 In particular, this gives that $\hom_{\lmod} ( \homol_0(b,- ), \homol_0 (a,-)) \cong \homol_0 (a,b)$ for $a, b \in \nat$ and these isomorphisms are compatible with composition. The result then follows by standard arguments.
\end{proof}

This has the immediate consequence:

\begin{cor}
\label{cor:H1_outer}
Considered as a left $\catlie$-module, $\homol_1$ belongs to $\flie^\mu$.
\end{cor}

\begin{rem}
The results of  \cite{P_outer} explain how $\flie^\mu$ is related to the category $\f^{\mathrm{Out}} (\gr\op)$ of outer functors on free groups (as introduced in \cite{PV}). Via this relationship, the objects $\homol_0$ and $\homol_1$ are directly related to the higher Hochschild homology functors considered in \cite{PV}, inspired by the work of Turchin and Willwacher \cite{MR3982870}. 

It was Turchin and Willwacher who observed that inner automorphisms of free groups act trivially upon the relevant higher Hochschild homology groups.  As observed in \cite{PV}, they are {\em outer functors} on free groups. This is {\em equivalent} to the assertion that both $\homol_0$ and $\homol_1$ lie in $\flie^\mu$. 

Whilst the fact that $\homol_0$ lies in $\flie^\mu$ is clear (as exploited in the proof of Theorem \ref{thm:fliemu_homol0} above), hitherto there was no direct argument establishing that $\homol_1$ also lies in $\flie^\mu$. Corollary \ref{cor:H1_outer} provides such an argument.  
\end{rem}

\subsection{A complex for left $\catlie$-modules}

In \cite{P_outer}, the left adjoint to $\flie^\mu \hookrightarrow \flie$ was identified as the functor $F \mapsto F^\mu$, where 
$F^\mu$ is defined as the cokernel of the natural transformation 
\[
\mut : \delta F \rightarrow F.
\]

\begin{nota}
Write $\lHmod$ for the category of left $\homol_0$-modules, so that there is an inclusion 
$ 
\lHmod \hookrightarrow \lmod
$
given by restriction along $\catlie \rightarrow \homol_0$, 
\end{nota}

Using the equivalence of Theorem \ref{thm:fliemu_homol0}, the left adjoint to $\flie^\mu \hookrightarrow \flie$ is identified by the following:

\begin{prop}
\label{prop:induction}
The left adjoint to the functor $\lHmod \hookrightarrow \lmod$ is 
\[
\homol_0 \otimes_\catlie - : \lmod \rightarrow \lHmod.
\]
\end{prop}

\begin{proof}
This is simply the fact that induction is left adjoint to restriction. 
\end{proof}

\begin{nota}
Write $\dgmod$ for the category of left $\dgcat$-modules in chain complexes. Objects are (homologically graded) chain complexes of left $\rat\fb$-modules $C_\bullet$, equipped with a structure morphism 
\[
\dgcat \otimes_\fb C_\bullet \rightarrow C_\bullet
\]
that satisfies the unit and associativity axioms. Morphisms are morphisms of chain complexes of left $\rat\fb$-modules (of degree zero) that are compatible with  the respective structure maps. 
\end{nota}

Clearly there is an induction functor lifting that of Proposition \ref{prop:induction} to the complex level: 
\[
\dgcat \otimes_{\catlie} - : \lmod \rightarrow \dgmod.
\]
Explicitly, for a left $\catlie$-module, $\dgcat \otimes_\catlie M$ is the complex concentrated in homological degrees $1$ and $0$:
\[
\big (\delta^{(1)} \catlie \big) \otimes _\catlie M 
\rightarrow 
M 
\]
with differential induced by $\mut$. This should be viewed as a refinement of the `complex' $\mut : \delta F \rightarrow F$ used in defining $F \mapsto F^\mu$ (using the notation of \cite{P_outer}). 

\begin{rem}
One can go further and consider applying $\dgcat \otimes _\catlie  -$ to complexes of left $\catlie$-modules.
However, since the degree one term of $\dgcat$ is not projective as a right $\catlie$-module, this will not behave well homologically. Instead, by Theorem \ref{thm:ce_dgcat}, one should use the `resolution' $\ce_\lie \twoheadrightarrow \dgcat$ (considered only as $\catlie$-bimodules and with projectivity only concerning the right $\catlie$-module structure).
\end{rem}

\begin{prop}
\label{prop:properties_dgcat_otimes}
For $M$ a left $\catlie$-module, 
\begin{enumerate}
\item 
the homology $H_\epsilon (\dgcat \otimes_\catlie M)$ is naturally a left $\homol_0$-module for $\epsilon \in \{0, 1 \}$; 
\item 
there is a natural isomorphism $H_0 (\dgcat \otimes_\catlie M) \cong \homol_0 \otimes_\catlie M$; 
\item 
the $\dgcat$-module structure of $\dgcat \otimes_\catlie M$ induces a morphism of left $\homol_0$-modules:
\[
\homol_1 \otimes_{\homol_0} H_0 (\dgcat \otimes_\catlie M) 
\cong 
\homol_1 \otimes_\catlie M 
\rightarrow 
H_1 (\dgcat \otimes_\catlie M).
\]
\end{enumerate}
\end{prop}

\begin{proof}
These statements are proved by a straightforward generalization of the properties of the homology of a DG module over a DG algebra. 
\end{proof}

\begin{exam}
\label{exam:lie_alg_hom_vs_dgcat_hom}
For a Lie algebra $\g$, one has the associated left $\catlie$-module $\underline{\g}$. Hence one can form the left $\dgcat$-module $\dgcat \otimes_\catlie \underline{\g}$. When $\g = \lie(V)$, this simply gives $\dgcat \otimes _\fb \underline{V}$, which is described explicitly above. In the general case,  the underlying complex has the form 
$
\delta^{(1)} \catlie \otimes_{\catlie} \underline{\g} \rightarrow \underline{\g}.
$ 

Corollary \ref{cor:nat_trans_ce_dgcat_homol} (with $M= \underline{\g}$) gives the natural transformation of (graded) left $\catlie$-modules:
\[
H_* (\g; \underline{\g}) \rightarrow H_* (\dgcat \otimes_\catlie \underline{\g}).
\]
This is an isomorphism in degree zero, where one has  $H_0 (\g; \underline{\g})$, which is an $\homol_0$-module. 

In degree one, there is a surjection of left $\catlie$-modules, 
\[
H_1 (\g; \underline{\g}) \twoheadrightarrow H_1 (\dgcat \otimes_\catlie \underline{\g})
\]
and the codomain is an $\homol_0$-module, by Proposition \ref{prop:properties_dgcat_otimes}. If $\g$ is the free Lie algebra $\lie (V)$, this map is an isomorphism; it is not clear how close this map is to being an isomorphism for general $\g$.

The above surjection factors to give a natural surjection of left $\homol_0$-modules:
\[
\homol_0 \otimes_{\catlie} H_1 (\g; \underline{\g}) 
\twoheadrightarrow 
H_1 (\dgcat \otimes_\catlie \underline{\g}).
\]
This exhibits $H_1 (\dgcat \otimes_\catlie \underline{\g})$ as an approximation to $\homol_0 \otimes_{\catlie} H_1 (\g; \underline{\g}) $.
\end{exam}

\section{$\homol_1$ as a syzygy of $\homol_0$}
\label{sect:syzygy}

This short section serves to explain the relationship between $\homol_1$ and $\homol_0$ that is obtained when one forgets structure and considers these as left $\catlie$, right $\rat\fb$-modules. This interpretation is implicit in \cite{PV}.

\begin{thm}
\label{thm:syzygy}
The underlying complex of $\dgcat$ provides a projective resolution of $\homol_0$ in left $\catlie$, right $\rat\fb$-modules.

Hence, the exact sequence 
\[
0
\rightarrow
\homol_1 
\rightarrow
\catlie \odot I
\rightarrow 
\catlie 
\rightarrow 
\homol_0
\rightarrow 
0
\]
restricted to left $\catlie$, right $\rat\fb$-modules exhibits $\homol_1$ as the second syzygy of $\homol_0$.
\end{thm}

\begin{proof}
Since we are working in characteristic zero, the group rings $\rat \sym_n$ are all semisimple. Hence, for the projectivity statement, it suffices to consider the left $\catlie$-module structure. 

Now, $\catlie (n,-)$ is projective as a left $\catlie$-module for each $n$, as follows from Yoneda's lemma (see \cite{P_analytic}, for example). Moreover, $\big(\catlie \odot I\big) (n,-)$ can be written, as a left $\catlie$-module, as the direct sum of copies of $\catlie (n-1,-)$ indexed by subsets of $\mathbf{n}$ of cardinal $n-1$. Thus, this also is projective as a left $\catlie$-module. 

This establishes the first statement, from which the second follows essentially by definition of the second syzygy.
\end{proof}

\begin{rem}
The results of \cite{PV} show more: namely, the resolution of $\homol_0$ provides by $\dgcat$ is very close to being a minimal resolution. The failure to be minimal stems from the exceptional behaviour of the family of simple $\catlie$-modules corresponding to the trivial representation of $\rat \sym_n$, for each $n \in \nat$. These are {\em projective} $\catlie$-modules and hence are also projective in $\flie^\mu$. In particular, they do not require `resolving', whereas $\dgcat$ contains contributions corresponding to these, hence is not minimal. This family is, however, the only obstruction to minimality.
\end{rem}


\providecommand{\bysame}{\leavevmode\hbox to3em{\hrulefill}\thinspace}
\providecommand{\MR}{\relax\ifhmode\unskip\space\fi MR }
\providecommand{\MRhref}[2]{%
  \href{http://www.ams.org/mathscinet-getitem?mr=#1}{#2}
}
\providecommand{\href}[2]{#2}

\end{document}